\documentclass[a4paper,twoside,12pt]{amsart}      
\usepackage{amsmath,amssymb,amsfonts} 
\usepackage{amsthm}
\usepackage{ulem}
\usepackage{graphics}   
\usepackage{cancel}
\usepackage{color}                    
\usepackage{hyperref}                 
\usepackage{graphicx,eucal}
\usepackage{latexsym,epsfig,epic,eepic}
\usepackage[printwatermark]{xwatermark}
\usepackage{xcolor}
\usepackage{graphicx}
\usepackage{lipsum}

\newcommand{\stkout}[1]{\ifmmode\text{\sout{\ensuremath{#1}}}\else\sout{#1}\fi}

\newcommand{\Z}{\mathbb{Z}}

\newcommand{\R}{\mathbb{R}}
\newcommand{\C}{\mathbb{C}}
\newcommand{\Pone}{\mathbb P ^1}
\newcommand{\Per} {\operatorname{Per} }
\newcommand{\percompact}{{\mathcal Per}}
 
\newcommand{\im}{\text{Im}}

\newtheorem{theorem}{Theorem}[section]
\newtheorem{proposition}[theorem]{Proposition}
\newtheorem{corollary}[theorem]{Corollary}
\newtheorem{lemma}[theorem]{Lemma}

\theoremstyle{definition}
\newtheorem{definition}[theorem]{Definition}
\newtheorem{remark}[theorem]{Remark}

\newcommand{\per}{\operatorname{Per}}

\newcommand{\Torelliformg}{ \Omega^\pm\mathcal{S}_{g,\text{\sout{2}}}}
\newcommand{\BOmega}{\Omega^{\pm,*}_0}

\newcommand{\mnc}{\overline{\mathcal{S}}}
\newcommand{\mrs}{\mathcal{S}}
\newcommand{\curvesystem}{\bf{c}}

\title{Isoperiodic meromorphic forms: two simple poles}

\author{Gabriel Calsamiglia}
\address{Departamento de Matem\' atica Aplicada, Universidade Federal Fluminense,Rua Professor Marcos Waldemar de Freitas Reis, s/n, Campus Gragoat\' aÂ¡, Bloco G, Niter\' oi, 24210-201 Brazil} 

\email{gcalsamiglia@id.uff.br}

\author{Bertrand Deroin}

\address{CNRS \& Universit\'e Cergy-Pontoise, UMR 8088, 2 av. Adolphe Chauvin, 95302 Cergy-Pontoise Cedex, France}

\email{bertrand.deroin@u-cergy.fr}



\date{2021}

\begin{document}

\maketitle
\begin{abstract}
    In this paper we prove that isoperiodic moduli spaces of meromorphic differentials with two simple poles on homologically marked smooth curves are non empty and connected, unless they correspond to  double covers of $\mathbb{C}/\mathbb{Z}$ on curves of genus at least two. We deduce dynamical consequences for the corresponding isoperiodic foliation. 
\end{abstract}
\section{Introduction}
\subsection{Overview}  This work is the second of a series  aiming at  study the dynamics and topology of the isoperiodic foliations on moduli spaces of meromorphic one forms on algebraic curves. The isoperiodic equivalence relation is defined locally by the condition that the integrals on \textit{every} closed cycle in the domain where the form is holomorphic remain constant. 

Moduli spaces of meromorphic forms with simple poles appear as boundary components in the Deligne-Mumford compactification of moduli spaces of holomorphic one forms (see  \cite{5A18, 5A19}). The understanding of these foliations will complete the picture of the dynamics of isoperiodic foliations on moduli spaces of stable forms given in \cite{CDF2}. 

We consider here the case of moduli spaces of forms on curves of any genus with two poles, both simple. It will constitute the initial step for an inductive argument to similar statements in any number of simple poles (see work in progress in collaboration with Liza Arzakhova in \cite{ACD}).

There is a conceptual reason to treat this case separately: up to scaling, any form with two simple poles has one of residue $+1$ and the other of residue $-1$ and therefore the integral along a closed cycle   defines an element in the topological group $\mathbb{C}/\mathbb{Z}$ that depends only on its homology class of the \textit{closed} surface.


In the case of holomorphic one forms on compact curves, the dynamics of the isoperiodic  foliation is related to the topology of the period mapping, defined on some Torelli cover of the moduli space (see \cite{McMullen, CDF2}). In the present paper, we show that, apart from isoperiodic sets consisting of double covers of a bi-infinite cylinder $\mathbb{C}/\mathbb{Z}$ of genus at least two, the isoperiodic moduli spaces of homologically marked meromorphic differentials with two simple poles are non empty and connected. Related results have been proven recently, see \cite{McMullen, CDF2,Winsor}. As in \cite{CDF2}, this allows to show that the fundamental group of the leaves of the isoperiodic foliation surjects onto the stabilizer in \( \text{Sp} (2g, \mathbb Z) \) of their associated period. Another consequence of interest for the dynamics is that we can transfer the dynamical properties of the action of the group \( \text{Sp} (2g, \mathbb Z) \) on \( (\mathbb C /\mathbb Z)^{2g} \) to properties satisfied by the isoperiodic foliation.

As an application, we compute,  using Ratner's theory ( specifically \cite{G}),  the closed subsets of moduli spaces that are saturated/invariant by the isoperiodic foliation, and find that they are real analytic submanifolds. A qualitative difference with the case of holomorphic differentials is that some transcendental leaves are closed in these quasi-projective manifolds. They correspond to forms whose periods describe a given discrete lattice in $\mathbb{C}$. Such phenomenon occurs for some algebraic foliations, for instance, the one defined by \( dy / dx = y \) in \(\mathbb C^2 \), but it seems to be quite rare to have a transcendental leaf of an algebraic foliation which is closed in a Zariski open subset. See for instance \cite{Camacho/LinsNeto/Sad, Camacho/Scardua} or the work of Cousin on the Garnier system \cite{Cousin}. We note that these leaves are examples of affine  manifolds (in the sense of \cite{Moller_linear}) that are not algebraic. Bakker and Mullane have such an example in moduli spaces with marked points. We do not know an example of a non algebraic affine manifold which is also invariant by the \(SL(2,\mathbb R)\)-action in the moduli space of meromorphic forms (such examples do not exist in moduli spaces of abelian differentials by a result of Filip, see \cite{Filip}).      

Another consequence we derive is the ergodicity with respect to Lebesgue measure of the restriction of the isoperiodic foliation to every closed saturated subset. A particularly interesting closed saturated subset that does not arise in the moduli space of abelian differentials, is the moduli space of meromorphic differentials with two poles and \textit{real periods}.  For every genus $g$ this set identifies with a copy of moduli space \(\mathcal M_{g,2}\), and plays a key role in the alternative proof \cite{GK1} by Grushevsky and Krichever of Diaz' theorem (the non existence of a complete subvariety of dimension \(\geq g-1\) in moduli space \(\mathcal M_{g} \) of curves of genus \(g\)), and in its generalization by Krichever \cite{Krichever} that provides a solution  of Arbarello's conjecture, stating that a compact holomorphic submanifold of \(\mathcal M_{g} \) of dimension \(g-n\) intersects the Weierstrass locus of curves admitting  a meromorphic function with a single pole of order at most \(n\). Recently, Krichever, Lando and Skripchenko proved that the isoperiodic foliation on the moduli space of meromorphic differentials with a single double pole is ergodic, by studying the combinatorics of cut diagrams, see \cite{KLS}. Ergodicity of isoperiodic foliations on various moduli spaces of holomorphic differentials have been established recently, see \cite{McMullen, Hamenstadt, CDF2, Ygouf, CEW}.



\subsection{Statement of results}

Let $\mathcal{M}_{g,n}$ denote the moduli space of smooth genus $g$ complex curves with $n$ marked ordered points. For $g\geq 1$ consider the bundle $\Omega^{\pm}\mathcal{M}_{g,2}\rightarrow \mathcal{M}_{g,2}$ whose fiber over a point $(C,x_-,x_+)$ is the space of meromorphic forms $\omega$ on $C$ having  simple poles on $x_-,x_+$ with residues $-1,+1$ respectively. It is an affine bundle directed by the Hodge (vector) bundle  $\Omega\mathcal{M}_{g,2}\rightarrow\mathcal{M}_{g,2}$ of holomorphic one forms, hence a complex manifold of dimension $4g-1$. 

The choice of \((\pm 1)\)-residues at the two poles of $\omega$ allows to define, for each homological class of closed path $\gamma$ in \(C\), a class in $\mathbb{C}/\mathbb{Z}$, namely \[\big[\int_{\gamma}\omega\big]\in\mathbb{C}/\mathbb{Z}.\]
A deformation of forms $\{\omega_t\}_{t\in[0,1]}$ in $\Omega^\pm\mathcal{M}_{g,2}$ is said to be isoperiodic if the elements in $\mathbb{C}/\mathbb{Z}$ defined above are constant for any choice of homology class. The main aim of this paper is the study of the topological and dynamical properties of the decomposition of the total space in isoperiodic equivalence classes, the so-called isoperiodic foliation.

To be able to compare homology classes on different curves we need to identify their homology groups. It is well known that the orbifold universal cover  $\mathcal{T}_{g,n}\rightarrow \mathcal{M}_{g,n}$ of the moduli space is biholomorphic to the Teichm\" uller space \(\mathcal{T}_{g,n}\). The fiber over a point \((C,(q_1,\ldots,q_n))\in\mathcal{M}_{g,n}\) is  given by the different possible identifications of the fundamental group of $C$ with the fundamental group of a reference surface of genus $g$ with $n$ marked ordered points  $\Sigma_{g,n}$. In this representation, the covering group is the mapping class group $\text{Mod}(\Sigma_{g,n})$ of isotopy classes of diffeomorphisms of $\Sigma_g$ that fix each marked point. 

Let \(\Sigma_{g,2}=(\Sigma_g,p_-,p_+)\) be an oriented closed surface of genus \(g\) with two ordered distinct marked points. Since we only want to keep the information of the identification at the homology level of the curve without the information on the marked points, we consider the quotient of $\mathcal{T}_{g,2}$ by the subgroup $\mathcal{I}(\Sigma_{g,\text{\sout{2}}})\subset\text{Mod}(\Sigma_{g,2})$ formed by the elements that act trivially on \(H_1(\Sigma_g,\mathbb{Z})\) \footnote{The notation is chosen to distinguish it from other known subgroups of $\text{Mod}(\Sigma_{g,2})$ characterized by the (trivial) action on some other homology groups:  $\mathcal{I}(\Sigma_{g,2})$ is the subgroup that fixes every class in the relative homology group $H_1(\Sigma_g,p_-,p_+;\mathbb{Z})$ and $\mathcal{I}(\Sigma_{g,2^*})$ the subgroup that acts trivially on the punctured surface homology group $H_1(\Sigma_g\setminus\{p_-,p_+\},\mathbb{Z})$}. We denote by \[\mathcal{S}_{g,\text{\sout{2}}}=\frac{\mathcal{T}_{g,2}}{\mathcal{I}(\Sigma_{g,\text{\sout{2}}})}\] the quotient space. To each point in it there corresponds a  tuple \((C,x_-,x_+,m)\) where \((C,x_-,x_+)\)  is a smooth complex curve of genus \(g\) with two marked points and  \( m : H_1 (\Sigma_g, \Z) \rightarrow H_1 (C, \Z)\) is an isomorphism, up to the equivalence \((C,x_-,x_+, m) \sim (C',x'_-,x'_+, m') \) if there exists a biholomorphism \( h : (C,x_-,x_+)\rightarrow (C',x'_-,x'_+) \) such that \( m' = h_* \circ m\). The mapping class group of $\Sigma_{g,2}$ acts on \(\mathcal{S}_{g,\text{\sout{2}}}\) by precomposition on the marking and the quotient map gives a covering map 
\begin{equation} \label{eq:torellicover}
    \mathcal{S}_{g,\text{\sout{2}}}
\rightarrow\mathcal{M}_{g,2}
\end{equation} with covering group $\text{Sp}(2g,\mathbb{Z})$.
The bundle $\Omega^{\pm}\mathcal{M}_{g,2}\rightarrow \mathcal{M}_{g,2}$ can be pulled back to $\mathcal{S}_{g,\text{\sout{2}}}$ by using the map \eqref{eq:torellicover} producing a bundle $\Omega^{\pm}\mathcal{S}_{g,\text{\sout{2}}}\rightarrow \mathcal{S}_{g,\text{\sout{2}}}$ whose fiber over a point $(C,x_-,x_+,m)$ is the vector space $\Omega^{\pm}C$ of meromorphic forms $\omega$ on $C$ having two simple poles of residues $-1$ and $+1$ at $x_-$ and $x_+$ respectively.  An element in $\Omega^{\pm}\mathcal{S}_{g,\text{\sout{2}}}$ will sometimes be denoted as a triple $(C,m,\omega)$, since the pole information is already in $\omega$.

\begin{definition}
The period of $(C,m,\omega)\in \Omega^\pm\mathcal{S}_{g,\text{\sout{2}}}$ is the homomorphism \[\per (C,m,\omega): H_1(\Sigma_g,\mathbb{Z})\rightarrow \mathbb{C}/\mathbb{Z}\text{   defined by}\] \[\gamma\mapsto\big[\int_{m(\gamma)}\omega\big].\] 
The period map $\per_{g}$ on \(\Omega^\pm\mathcal{S}_{g,\text{\sout{2}}}\) is the map   \[\per_g:\Omega^\pm\mathcal{S}_{g,\text{\sout{2}}}\rightarrow H^1(\Sigma_g,\mathbb{C}/\mathbb{Z}) \]defined by \((C,m,\omega)\mapsto \per(C,m,\omega)\in\text{Hom}(H_1(\Sigma_g,\mathbb{Z});\mathbb{C}/\mathbb{Z})\simeq H^1(\Sigma_g,\mathbb{C}/\mathbb{Z}) \).
\end{definition}

In Proposition \ref{p:period map on smooth curves} we will prove that $\per_g$ is a holomorphic submersion. The underlying foliation is therefore regular and holomorphic. 

\begin{definition}
 The degree of a homomorphism \(p:H_1(\Sigma_g,\mathbb{Z})\rightarrow\mathbb{C}/\mathbb{Z}\) with image \(\Lambda\) is the cardinality of \(\Lambda\), that we denote \(|\Lambda|\). 
 
\end{definition}

\begin{remark}
If the period homomorphism of an element$ (C,m,\omega)\in\Omega^\pm\mathcal{S}_{g,\text{\sout{2}}}$ has finite degree $d<\infty$, then the map \(z\mapsto e^{\frac{1}{2\pi i}\int_{z_0}^{z}\omega}\) extends to a branched cover $C\rightarrow\mathbb{C}\cup\infty$ of the Riemann sphere of topological degree $d$. In particular if $d=1$ the only possibility is $g=0$. 
\end{remark}


This is actually the only obstruction.  
\begin{theorem}\label{t:realization}
Let $g\geq 1$. The fiber of the period map $\per_g$ over a non-trivial homomorphism is non-empty. 
\end{theorem}
Aiming to applying a Transfer Principle analogous to the case of holomorphic forms in  \cite{CDF2} we consider the problem of the connectedness of the fibers of $\per_g$: 
\begin{theorem}\label{t:Cgn}
Let $g\geq 1$. The fiber of the period map $\per_{g}$ over a homomorphism of degree at least three is connected. 
\end{theorem}
\begin{theorem}\label{t:disconnected fibers}
The fiber of $\per_{g}$ over a  degree two homomorphism is disconnected if $g\geq 2$ and connected if $g=1$.  
\end{theorem}

The mapping class group $\text{Mod}(\Sigma_{g,n})$ acts on source and target of the map $\per_g$ equivariatly, that is,  for every $(C,m, \omega)\in\Omega^\pm\mathcal{S}_{g,\text{\sout{2}}}$ and $\varphi\in\text{Mod}(\Sigma_{g,n})$, 
$$\per_g(\varphi\cdot (C,m,\omega))=\per_g((C,m,\omega))\circ \varphi^{-1}_*.$$ The fibration induced by $\per_g$ thus induces a regular holomorphic foliation $\mathcal{F}_g$ on the quotient (moduli) space $\Omega^{\pm}\mathcal{M}_{g,2}$ called the isoperiodic foliation. Theorem \ref{t:Cgn} immediately implies 
\begin{corollary}\label{c:surjection of pi1 of leaf}
If $L\subset \Omega^{\pm}\mathcal{M}_{g,2}$ is the leaf of $\mathcal{F}_g$ associated to a period $p\in H^{1}(\Sigma_{g},\mathbb{C}/\mathbb{Z})$ of degree at least three, then $$\pi_{1}(L)\rightarrow\text{Stab}(p)\subset \text{Aut}(H_{1}(\Sigma_{g}))\text{  is surjective. }$$ 
\end{corollary}

There are a few invariants that are well defined and constant along the leaves of $\mathcal{F}_g$. 

On the set of connected fibers of $\per_g$ we can apply the Transfer Principle to deduce results on the dynamics of $\mathcal{F}_g$. For this purpose, let us analyze some invariants of the mapping class group action on $H^1(\Sigma_g, \mathbb{C}/\mathbb{Z})$. 
\begin{definition}
 Given a point $(C,\omega)\in\Omega^{\pm}\mathcal{M}_{g,2}$,  denote by $\Lambda_{\omega}\subset\mathbb{C}/\mathbb{Z}$ the subgroup formed by the periods of $\omega$ on cycles of $C$.  
 \end{definition}
 The map $(C,\omega)\mapsto\Lambda_{\omega}$ is constant along a leaf $L$ of $\mathcal{F}_g$ and has image $\Lambda_L$. 
 
 Given a subset or subgroup $\Lambda\subset\mathbb{C}/\mathbb{Z}$ we denote $\Im{\Lambda}\subset\mathbb{R}$ its projection in the imaginary axis.
 \begin{definition}
 If $p:H_1(\Sigma_g,\mathbb{Z})\rightarrow \mathbb{C}/\mathbb{Z}$ is a homomorphism with discrete imaginary part $\Im (p(H_1(\Sigma_g,\mathbb{Z})))=\alpha\mathbb{Z}\subset \mathbb{R}$, the product  $\Re(\widetilde{p})\cdot \Im (p)$ in $H^1(\Sigma_g,\mathbb{R})$ defines an element  $\deg_{\alpha}(p)\in \mathbb{R}/\alpha\mathbb{Z}$ that does not depend on the choice of lift $\widetilde{p}$ of $p$ to $\mathbb{C}$. Furthermore, $\deg_{\alpha}$ is invariant under the action of the mapping class group of $\Sigma_g$ on the subset of  $H^1(\Sigma_g,\mathbb{C}/\mathbb{Z})$ with imaginary part $\alpha\mathbb{Z}$.
 \end{definition}
The map $(C,\omega)\mapsto\deg_{\alpha}(\per (C,m,\omega))$ is well defined regardless of the choice of $m$, and a constant $\deg_{\alpha}(L)\in\mathbb{R}/\alpha\mathbb{Z}$ on each leaf $L$ satisfying $\Im\Lambda_L=\alpha\mathbb{Z}$.  

The next result classifies the closure of the leaves of the isoperiodic foliation. 
\begin{theorem}
\label{t:leaf closures}
Let $L$ be a leaf of $\mathcal{F}_g$ and $\Lambda=\overline{\Lambda}_L\subset\mathbb{C}/\mathbb{Z}$ denote the topological closure of $\Lambda_L$. Then the closure of $L$ in $\Omega^{\pm}\mathcal{M}_{g,2}$ is the subset of forms $(C,\omega)$ with $\Lambda_{\omega}\subset \Lambda$, unless $\Im{\Lambda}=\alpha\mathbb{Z}\neq 0$ is discrete non-trivial, in which case the following extra conditions must hold:  $\Im\Lambda_{\omega}=\alpha\mathbb{Z}$ and 
$\deg_{\alpha}(C,\omega)=\deg_{\alpha}(L)$.

In either case $\overline{L}$ is a real analytic subset and the restriction of $\mathcal{F}_g$ to it is ergodic. 
 The leaf $L$ is closed if and only if $\Lambda$ is discrete and algebraic if and only if $\Lambda$ is finite (it corresponds to a Hurwitz space of branched covers over the sphere of degree $|\Lambda|$). 
\end{theorem}

\subsection{Acknowledgements}
We would like to thank 
Liza Arzhakova, Corentin Boissy, Simion Filip, Selim Ghazouani, Igor Krichever, Sergei Lando, Frank Loray, Scott Mullane, Luc Pirio and Sasha Skripchenko for sharing their viewpoints on the subject with us. 
 We also thank the France Brazil agreement in Mathematics for the support to carry this project; IMPA (Rio de Janeiro), Universidad Federal Fluminense (Niter\'oi) and IMJ (Paris Sorbonne) for the stimulating working conditions. G.C. was partially supported by CNPq, Faperj and Capes (Brasil).

\section{Strategy of the proof of Theorem  \ref{t:Cgn}}
The proof follows by induction on the genus. 

The cases of genera one and two are treated separately by analytic methods and the Torelli map. Each fiber of $\per_1$ is biholomorphic to a Zariski open set of $\mathcal{T}_{1,1}$ and is therefore connected. On the other hand, each fiber of $\per_2$ is a branched double cover, over a Zariski open subset of the Siegel space $\mathfrak{S}_{2}$. The branch points correspond precisely to odd forms with respect to the hyperelliptic involution. In the case of a fiber over a degree two homomorphism, all forms are even and we manage to prove that the cover is disconnected. In the case of degree at least three, we manage to construct an example of branch point  (see subsection 
\ref{ss:genus 2}) to deduce the connectedness of the cover.

Fix some $g\geq 3$ and assume the inductive hypothesis, i.e. Theorem \ref{t:Cgn} is true up to genus $g-1$. Then run the following program for any given homomorphism $p\in H^1(\Sigma_g,\mathbb{Z})\rightarrow \mathbb{C}/\mathbb{Z}$ of degree at least three: 

{\bf 1. Bordify $\per^{-1}(p)$} (see Section \ref{s:extension of the period map and bordification}).  The Teichm\" uller space $\mathcal{T}_{g,2}$ admits a topological bordification $\overline{\mathcal{T}}_{g,2}$ formed by marked stable curves of genus $g$ with two marked points. It is stratified by the number of nodes of the underlying curves and each stratum is a complex manifold.  The action of $\text{Mod}(\Sigma_{g,2})$ preserves the stratification and the quotient is isomorphic to the Deligne-Mumford compactification $\overline{\mathcal{M}}_{g,2}$ of $\mathcal{M}_{g,2}$. The added points, called the boundary, form a normal crossing divisor in $\overline{\mathcal{M}}_{g,2}$ at the orbifold chart level.  Quotienting $\overline{\mathcal{T}}_{g,2}$ by  $\mathcal{I}(\Sigma_{g,\text{\sout{2}}})$ produces a  stratified bordification $\overline{\mathcal{S}}_{g,\text{\sout{2}}}$ of $\mathcal{S}_{g,\text{\sout{2}}}$. The bundle of stable  meromorphic forms with two poles over the Deligne-Mumford compactification of $\mathcal{M}_{g,2}$ can be pulled back to a bundle $\Omega^{\pm}\overline{\mathcal{S}}_{g,\text{\sout{2}}}$. 

Any stable form in the boundary has a non-trivial local isoperiodic deformation space. Two conditions that guarantee that this local isoperiodic deformation space leaves the boundary are that the form has no zero components and the residue of the form at each non-separating node is zero. The bordification of $\per^{-1}(p)$ that we are interested in is its closure in the space  $\Omega^{\pm,*}_{0}\overline{\mathcal{S}}_{g,\text{\sout{2}}}$ 
of forms in $\Omega^\pm\mathcal{S}_{g,\text{\sout{2}}}$  having zero residues at all non-separating nodes and no zero components. In fact the local isoperiodic deformation space projects to a smooth complex manifold in the orbifold charts of the moduli space transverse to each boundary component of  $\Omega\overline{\mathcal{M}}_{g,2}$ passing through the point (See Theorem \ref{t:local structure special cover} and the Appendix). The stratification of the boundary (defined by the number of nodes) induces a stratification of the bordification of the isoperiodic set and for each stratum of the ambient space passing through the point there is one isoperiodic component of the stratum that lies in it.  Around a point having only separating nodes, the local picture of the stratification is that of a normal crossing divisor. The picture changes by an  abelian ramified cover over the divisor when the curve underlying the form has at least one non-separating node (see the local model of this local branched cover in \cite{CDF2}[Section 4.4]).  The abelian ramified cover does not brake a nice property of the local stratification of a normal crossing divisor: a point in the codimension $k\geq 1$ stratum lies at the intersection of the closure of $k$ codimension one (local) connected components of the divisor. Any other local connected component of a non-open stratum accumulating the point has codimension $1\leq l\leq k$ and is \textit{precisely} the set of points that belong to the closure of $l$ of the $k$ codimension one components accumulating the point, and not more.  Moreover, the open stratum is locally connected at the point. 

To any stratified space  $X$ that is a locally abelian ramified cover over a normal crossing divisor we can define its dual boundary graph $\mathcal{C}(X)$ . It has a vertex for each (global) connected component of the codimension one stratum, and a simplex between $k$ vertices for each connected component of the codimension $k$ stratum lying in the closure of the corresponding $k$ components. It is well known that the boundary complex associated to $\overline{\mathcal{T}}_{g,2}$ is isomorphic to  the \textit{curve complex}  $\mathcal{C}_{g,2}$ on the genus $g$ compact surface with two marked points $\Sigma_{g,2}$ (see \cite{FM}{Chapter 4.1}). 

The closure of $\per^{-1}(p)$ in $\Omega^{\pm,*}_{0}\overline{\mathcal{S}}_{g,\text{\sout{2}}}$ is also shown to be  stratified and a locally abelian ramified cover  over a normal crossing divisor. In particular,  the connectedness of  $\per^{-1}(p)$ is equivalent to that of its bordification. 

Moreover, the transversality condition allows to define a continuous map of complexes  \begin{equation}\label{eq:complex inclusion}\mathcal{C}(\per^{-1}(p))\rightarrow \mathcal{C}(\Omega^{\pm}\overline{\mathcal{S}}_{g,\text{\sout{2}}})\end{equation} that associates to each component of an isoperiodic stratum the component of the ambient stratification where it sits.
There is a subfamily of components of the codimension one stratum of the ambient space $\Omega^{\pm}\mathcal{S}_{g,\text{\sout{2}}}$ where we will be able to prove inductively that there is a single connected isoperiodic component of the isoperiodic stratum of codimension one associated to $p$. These are the so called $p$-simple boundary components and they correspond to forms over stable curves with one node that leaves a pole on each side, and moreover the form restricted to each part has degree at least three.  They define a subfamily of vertices of $\mathcal{C}(\per^{-1}(p))$ that span the subcomplex $\mathcal{C}'(\per^{-1}(p))$ of $p$-simple boundary points. 

{\bf 2. Isoperiodic degeneration towards boundary points }(see Section \ref{s:degenerating to a simple boundary point}). We will first prove that any point in $\per^{-1}(p)$ can be isoperiodically deformed in $\Omega_{0}^{\pm*}\overline{\mathcal{S}}_{g,\text{\sout{2}}}$ to a point belonging to a $p$-simple boundary component. In this step we use Schiffer variations, a way of deforming the singular flat metric undelying a stable meromorphic form with isolated zeros without changing the associated period homomorphism. To achieve this step we first degenerate to any boundary point, and then prove that any boundary point can be joined to a $p$-simple boundary point.

{\bf 3.  Connectedness of the boundary.}  This will be achieved by  showing  that the complex $\mathcal{C}'(\per^{-1}(p))$ is connected.  The inductive hypothesis of Theorem \ref{t:Cgn} allows to prove that the restriction of the map \eqref{eq:complex inclusion} to the subcomplex $\mathcal{C}'(\per^{-1}(p))$ of $p$-simple boundary points is injective at the level of the vertices. We will prove that it has connected image under the map \eqref{eq:complex inclusion}. The fact that there are only two poles allows to  rephrase the problem in algebraic terms: the complex $\mathcal{C}(\Omega^{\pm}\overline{\mathcal{S}}_{g,\text{\sout{2}}})$ is isomorphic to  $\mathcal{C}(\overline{\mathcal{S}}_{g,\text{\sout{2}}})$  which in its turn is isomorphic to  the quotient \begin{equation}\label{eq:quotient curve complex} \frac{\mathcal{C}_{g,2}}{\mathcal{I}(\Sigma_{g,\text{\sout{2}}})}\end{equation}  of the curve complex $\mathcal{C}_{g,2}$ under the natural action of the group $\mathcal{I}(\Sigma_{g,\text{\sout{2}}})$. Each $k$-simplex in \eqref{eq:quotient curve complex} is characterized by a $\mathcal{I}(\Sigma_{g,\text{\sout{2}}})$-orbit of a family of $k$ disjoint essential simple closed curves $c=c_{1}\sqcup\ldots\sqcup c_{k}$ in $\Sigma_{g,2}$. Whenever every $c_{i}$ separates the (two!) ordered marked points (and therefore the surface), we can characterize the simplex by the ordered splitting   $H_{1}(\Sigma_{g,2})=V_{1}\oplus\cdots\oplus V_{k+1}$ into pairwise orthogonal symplectic submodules of rank at least two induced by the parts of $\Sigma_{g,2}\setminus c$ . We suppose that the first marked point belongs to $V_{1}$ and the other to $V_{k+1}$. The order of the rest of factors is determined by imposing that $V_{j}$ has a common boundary component with $V_{j-1}$ and another with $V_{j+1}$. 
 
The complex $\mathcal{C}'(\per^{-1}(p))$ has its vertices corresponding to separating essential simple closed curves that separate the marked points, any higher dimensional simplex in it corresponds to marked stable forms having only separating nodes that separate the marked points. Therefore the image of any $k$-simplex in $\mathcal{C}'(\per^{-1}(p))$  by  \eqref{eq:complex inclusion} is characterized by a splitting of $H_{1}(\Sigma_{g,2})=V_{1}\oplus\cdots\oplus V_{k+1}$  into $k+1$ pairwise orthogonal symplectic submodules of rank at least two. Moreover for the vertices, i.e. when $k=1$, the restriction $p_{|V_{i}}$ has at least three points in the image. The condition on the size of $p(V_{i})$'s can be rephrased as follows:

 \begin{remark}\label{rem:degreeatleastthree}
 A homomorphism $p$ having image in  $\mathbb{C}/\mathbb{Z}$ has at least three points in the image if and only if its composition with  $\mathbb{C}/\mathbb{Z}\rightarrow\mathbb{C}/\frac{1}{2}\mathbb{Z}$ is non-trivial.
 Call this composition $[p]$. \end{remark}
 
 \begin{definition}\label{def:p-admissible}
Given a homomorphism $p:V\rightarrow A$ from a symplectic unimodular module to an abelian group $A$ we define:
\begin{itemize}
    \item a decomposition $V=\oplus_iV_i$ into non-trivial symplectic unimodular submodules is said to be $p$-admissible if $p_{|V_i}\neq 0$ for all $i$. 
\item The graph of $p$-admissible decompositions has a vertex for every $p$-admissible decomposition with two factors, and an edge between the vertices corresponding to  $V_1\oplus V_2$ and $V_1'\oplus V_2'$  if there exists a $p$-admissible decomposition with more factors having one of the $V_i$'s as factor and also one of the $V_i'$ 's.
\end{itemize}
\end{definition}

A word of warning about this definition: we do not regard the order of the factors. 

\begin{proposition}\label{p:algebrization}
Let \(g\geq 3\) and \(p\in H^1(\Sigma_g,\C/\Z)\). If the graph of $[p]$-admissible decompositions is connected and Theorem \ref{t:Cgn} is true up to genus $g-1$ then the complex \(\mathcal{C}'(\per^{-1}(p))\) is connected.\end{proposition}
The proof of this proposition is contained in Section \ref{s:sufficient conditions}. 
The following theorem (proven in Section \ref{s:algebraic results} ) allows to use the previous proposition to deduce the connectedness of $\per^{-1}(p)$. 

\begin{theorem}\label{t:graph of admissible}
 Given a non-trivial homomorphism $p:V\rightarrow A$ from a symplectic unimodular $\mathbb{Z}$-module $V$ of rank at least six to an abelian group $A$, the graph of  $p$-admissible decompositions is nonempty and connected.
\end{theorem}

A posteriori it would have been enough to bordify only by adding  the $p$-simple boundary points to prove the connectedness of the bordification. Unfortunately we have not found a proof of the degeneration part  that avoids passing through some other boundary components in $\Omega_{0}^{\pm*}\overline{\mathcal{S}}_{g,\text{\sout{2}}}$ as explained at the end of step 2. 

\section{Some preliminary tools}

\subsection{Criterion for \(p\)-admissible elements}
\begin{definition}\label{def:p-admissible element}
Let $p:V\rightarrow A$ be a non-trivial homomorphism from a symplectic unimodular $\mathbb{Z}$-module $V$ to an abelian group $A$. An element $v\in V$ is $p$-admissible if it belongs to a factor of a $p$-admissible decomposition.
\end{definition}

\begin{lemma} \label{l: admissible element}
Suppose $V$ has rank at least four and \(p:V\rightarrow A\) is non-trivial. A non zero element in \(V\) is  \(p\)-admissible  if and only if \(p\) does not vanish identically on the orthogonal of \(v\). In this case, it belongs to a factor of rank two of a \(p\)-admissible decomposition. In particular, the set \(NA(p)\subset V\) of non \(p\)-admissible elements is a submodule of rank at most one. \end{lemma}

\begin{proof}  Let \(v\) be a non zero element of \(V\). Assume \(v\) is \(p\)-admissible, namely \(v\in W_1\) where \( V= W_1\oplus W_2\) is a \(p\)-admissible decomposition. Then, the restriction of \(p\) to \(W_2\) does not vanish and \(W_2\subset v^\perp\) so \(p\) does not vanish on the orthogonal of \(v\). 

Reciprocally, assume that \(p\) is not identically zero on \(v^\perp\). We will assume that \(v\) is primitive, and consider a symplectic basis \(a_1, b_1, \ldots, \) such that \( v=a_1\). 

\vspace{0.2cm} 

\textit{1st case: \(p(a_1)\neq 0\).} In this case, we are done if \(p\) is not identically zero on \( (\Z a_1 \oplus\Z b_1)^\perp\) since in this case we can take \(W_1=\Z a_1 \oplus\Z b_1\) and \( W_2 = W_1^\perp\).  If \(p\) is identically zero on \((\Z a_1 \oplus \Z b_1)^\perp= \Z a_2\oplus \Z b_2\oplus \ldots\), we define  \( W_1= \Z  a_1 \oplus \Z ( b_1-b_2) \) and \(W_2=W_1^\perp\). The symplectic decomposition \(V=W_1\oplus W_2\) is \(p\)-admissible in this case since \(p(a_1)\neq0\), \(p(a_2+a_1)\neq 0\), and \( a_1\in W_1\), \(a_2+a_1\in W_1^\perp\). We are done since \(v=a_1\in W_1\).

\vspace{0.2cm} 

\textit{2d case: \(p(a_1)=0\).} In this case, \(p\) does not vanish on \(\Z a_2\oplus \Z b_2+\ldots\) since otherwise it would vanish on \(a_1^\perp=v^\perp\), and up to changing the basis \(a_2,b_2, \ldots, \) we can assume that \(p(a_2)\neq 0\).  If \(p(b_1)\neq 0\) we are done by setting \( W_1 = \Z a_1 +\Z b_1\) and \( W_2 = W_1^\perp \). If \(p(b_1)=0\), we consider \( W_1 = \Z a_1 \oplus  \Z (b_1+a_2)\) and \( W_2= W_1^\perp \) and we are done since \( b_1 +a_2 \in W_1 \setminus \text{ker} (p)\) and \( a_2 \in W_2 \setminus \text{ker} (p)\). So we conclude that \(v\) belongs to a \(p\)-admissible submodule of rank two. 

To end the proof of the Lemma, notice that if two non zero elements \(v, v'\) are not \(p\)-admissible, then both \( v^\perp\) and \((v')^\perp \) are contained in \(\text{ker} (p)\). Being corank one primitive submodules, they need to be equal since otherwise they would generate the whole \(V\), and so \( p\) would vanish identically. In particular, this proves that \( v\) and \(v'\) are rationally colinear, and the Lemma follows. \end{proof}
\begin{corollary}\label{c:padmissible decompositions exist}
Let $V$ be a symplectic unimodular submodule of rank at least four. Then, for any non-trivial homomorphism $p:V\rightarrow A$ there exists a $p$-admissible decomposition. 
\end{corollary}

\subsection{Preliminaries on $\per_g$} \label{ss: Preliminaries}
We begin by proving

\begin{proposition}\label{p:period map on smooth curves}
The period map $\per_{g}$ on \(\Omega^\pm\mathcal{S}_{g,\text{\sout{2}}}\) is a holomorphic submersion. 
\end{proposition}
\begin{proof}
The statement can be checked locally. By the description of $\Omega^\pm\mathcal{S}_{g,\text{\sout{2}}}$ as affine bundle over the Hodge bundle, the map $\per_g$ is, up to an adequate choice of local affine coordinate, defined by a germ of the map $$\Omega\mathcal{S}_g\rightarrow \text{Hom}(H_1(\Sigma_g,\mathbb{Z});\mathbb{C})\text{ given by }(C,m,\omega)\mapsto\{\gamma\mapsto\int_{m(\gamma)}\omega\}$$ where $\Omega\mathcal{S}_g\rightarrow \mathcal{S}_g$ denotes the Hodge bundle of \textit{holomorphic} one forms over the Torelli cover $\mathcal{S}_g\rightarrow \mathcal{M}_g$. This map is holomorphic everywhere and submersive around any non-zero form (see e.g. \cite{CDF2}). By appropriately choosing the base point of the affine model, we get the desired result. 
\end{proof}


\begin{proposition}\label{p: lifts}
Given \( p \in H^1 (\Sigma_g, \C/\Z)\setminus H^1 (\Sigma_g, \R /\Z)\), the set of lifts \(P\in H^1 (\Sigma_g, \C) \) that are the periods of a holomorphic differential on a smooth homologically marked curve, and whose reductions modulo \(\Z\) is \( p\), is infinite countable. If \(p\in H^1(\Sigma_g, \R/\Z)\setminus 0\) the set is empty and if $p=0$ the only possibility is $P=0$. 
\end{proposition}

\begin{proof}
Let \(P\) be one such lift for any $p\in H^1(\Sigma_g,\C/\Z)$ and write \( P= u + i v \) with \( u, v \in H^1 (\Sigma_g, \R)\). A necessary and sufficient condition on \(P\) to be the period of a non-zero holomorphic one form on a homologically marked curve is that 

\vspace{0.2cm}

1. the symplectic product \( u\cdot v \) is positive, and 

\vspace{0.2cm}

2. in the case the image of \(P\) in \(\C\), considered as a map from \(H_1(\Sigma_g, \Z)\) to \(\C\), is discrete, that the symplectic product \(u\cdot v\) is strictly greater than the covolume in \(\C\) of the image of \(P\).  

\vspace{0.2cm}

See Otto Haupt's Theorem in \cite{Haupt}. Hence, in the case $p=0$ the only possibility is that the form is zero. The case where $p$ has values in $\mathbb{R}$  has no lift that are periods of holomorphic one-forms. Suppose \( p \in H^1 (\Sigma_g, \C/\Z)\setminus H^1 (\Sigma_g, \R /\Z)\). 

Let now \(P' = P + w \) another lift, where \( w\in H^1 (\Sigma _g, \Z)\), and \( P' = u'+ i v'\) where \(u '= u +w \) and \( v' = v\). Since \( u'\cdot v' = u\cdot v + w\cdot v \), and that \(v\neq 0\) by assumption (otherwise \(p\) would belong to \( H^1 (\Sigma_g, \R/\Z)\)), there is a half space in \( H^1 (\Sigma_g, \Z)\) of choices of \(w\) so that \(P'\) satisfies the volume positivity condition 1. of Haupt's criterion.

Suppose now that for one \(P'\), say \(P\), the first condition is satisfied but not the second. In that case, we know that the kernel of \( P\) is a symplectic unimodular submodule of \( H_1 (\Sigma_g, \Z) \) of rank \(2g-2\), or equivalently that there exists two elements \( a  , b \in H^1 (\Sigma_g, \Z) \) such that \( a \cdot b = 1\) and two elements \( \alpha , \beta \in \C\) such that \(\Im (\beta \overline{\alpha} ) >0\) and \(P = a \alpha + b \beta\). Now let \( P'= P + w \).  Assume that \(P' \) is also of the form \( a' \alpha ' + b' \beta' \) with \( a' \cdot b' = 1\) and \( \Im (\beta' \overline{\alpha'} ) >0 \). We then have \( a' \alpha ' + b' \beta ' = a \alpha + b \beta + w\). This means that \(w\) is a rational combination of \( a\alpha \) and \( b \beta\). So, apart from a submodule of rank at most two, all the elements \(w\) in a half-space of \( H^1 (\Sigma_g, \Z) \) give rise to a period \(P' = P + w\) that satisfies conditions 1. and 2. 

The proposition follows. 
\end{proof}

Let $\mathcal{S}_{g,0}\rightarrow \mathcal{M}_{g,0}$ denote the Torelli cover of the moduli space of genus $g$ curves. Its covering group is the Torelli group $\mathcal{I}(\Sigma_{g,0})\subset \text{Mod}(\Sigma_{g,0})$ formed by elements that act trivially on the symplectic group $H_1(\Sigma_g)$. To each element in $\mathcal{S}_{g,0}$ we can associate a pair $(C,m)$ where $m$ is an isomorphism from $H_1(\Sigma_g)$ to $H_1(C)$.  Denote by $\Omega\mathcal{S}_{g,0}\rightarrow\mathcal{S}_{g,0}$ the pull back of the Hodge (vector) bundle of abelian differentials over $\mathcal{M}_{g,0}$.  The period map is well defined on $\Omega\mathcal{S}_{g,0}$ and the restriction of the period map to a fiber of $\Omega\mrs_{g,0}\rightarrow \mrs_{g,0}$ is linear and injective (recall that on a fixed compact Riemann surface a holomorphic one-form is completely determined by its period homomorphism). Therefore, the restriction of the bundle projection to a fiber of $\per$ is a biholomorphism onto its image. The next Proposition describes the image of this projection by using the injectivity of the  Torelli map $$J:\mrs_{g,0}\rightarrow \mathfrak{S}_g$$ to Siegel space of symmetric complex $g\times g$ matrices with positive definite imaginary part defined as follows: consider a symplectic basis \(a_1,b_1,\ldots, a_g,b_g\) 
of \(H_1 (\Sigma_g , \Z) \), i.e. a basis such that the only non-zero symplectic products of its elements are \(a_i\cdot b_i=+1\) and \(b_i\cdot a_i=-1\). For each homologically marked compact type genus two curve \( (C, m) \), let \(\omega_1, \ldots, \omega_g \) be a basis of the space of holomorphic \(1\)-forms on \(C\) such that \( \int _ {m(a_i)} \omega_j = \delta_{i,j}\), \(\delta_{i,j}\) being the Kronecker symbol. The Jacobian $g\times g$ matrix is  \( J (C) = (\int _{m(b_i)} \omega_j ) _{i,j}\). The fact that it belongs to the Siegel space \(\mathfrak S_g \) is a consequence of Riemann's relations.

\begin{proposition}[McMullen] \label{p: McMullen}
Given \( p \in H^1 (\Sigma_g, \C/\Z)\setminus H^1 (\Sigma_g, \R /\Z)\) for each lift 
\(P\in H^1 (\Sigma_g, \C) \) as in Proposition \ref{p: lifts},
 the set of elements of \(\mrs_{g,0} \) supporting a holomorphic form having period equal to \( P\) is biholomorphic (via $J$) to a slice of the image of $J$ by  by a linear Siegel subspace  \(\mathfrak S_P\subset \mathfrak{S}_g\) of genus $g-1$ .
\end{proposition}

\begin{proof}
Let \(P\) be a lift of \(p\). Suppose there exists a holomorphic \(1\)-form \(\omega_P\) on \( (C,m)\) such that \( (\int \omega) \circ m =P\). Since \(\omega_1,\ldots, \omega_g\) is a basis of \( \Omega (C)=H^0 (K_C) \), we need to have \( \omega_P= P(a_1)  \omega_1 +\ldots + P(a_g) \omega_g\) so that we get the relations 
\[ P(b_i) = P(a_1) J(C) _{i,1} +\ldots + P(a_g) J(C)_{i,g} \text{ for } i=1,\ldots, g.\] 
Reciprocally, these relations clearly imply the existence of a holomorphic \(1\)-form with period \(P\) on \((C,m)\). The set of Jacobians of homologically marked curves of compact type satisfying these relations is a copy of the genus $g-1$ dimensional Siegel space denoted \(\mathfrak S_P\) \end{proof} 

\section{Proof of Theorem \ref{t:Cgn} for $g=1,2$}\label{s:analytic methods for genus one and two} 

\subsection{Genus one}

\begin{lemma} \label{l: computation of the isoperiodic foliation in genus one}
Given an elliptic curve \(E\), and a morphism \( p: H_1 (E,\Z) \rightarrow \C / \Z \), there exists a unique meromorphic form on \(E\), up to translation, that satisfies 
\begin{enumerate} 
\item it is either holomorphic or it has two distinct simple poles, 
\item its period modulo \(\Z\) equals \(p\), 
\end{enumerate}
If \(p\) takes real values and is non vanishing, then we are always in the second case. 
\end{lemma} 

\begin{proof} To the period \(p\) one can associate the flat unitary line bundle \( (L, \nabla) \) having monodromy \( \exp (2i \pi p ): H_1(E, \Z) \rightarrow \C ^\star  \). Depending on whether this bundle is trivial or not, it has, up to multiplication by a non zero constant,  a unique holomorphic section, or a unique meromorphic section with a simple zero and a simple pole, which can be explicitly described by the quotient of two theta series. Denote this section \(s: E\rightarrow L\); the form defined by \( \nabla s = \omega s\) is either holomorphic or it has two distinct poles and satisfies \( p(\omega) = p\). 

In the case \(p\) is non zero but takes real values, the bundle \(L\) is never trivial, hence the form \(\omega\) cannot be holomorphic in that case.
\end{proof}

\begin{corollary}\label{c:genus one case}
A fiber of the period map $\per_1$ is biholomorphic to a Zariski dense subset of the Teichm\"uller space \( \mathcal T_{1,1}\) of genus one curves with a marked point. If the fiber is over a non zero homomorphism with real values, it is biholomorphic to \(\mathcal T_{1,1}\). In particular any fiber of $\per_1$ is connected. 
\end{corollary}

\begin{proof} 
This is an immediate corollary of Lemma \ref{l: computation of the isoperiodic foliation in genus one}, as soon as one sees that in a level set of the period map, the one corresponding to case (1) is a Zariski closed set of dimension zero. But this set is easily described: it corresponds exactly to the possible lifts in \( H^1 (E, \C) \)  of \(p\) having positive volume, see subsection \ref{ss: Preliminaries}. In the case where \(p\) takes real values, one cannot lift \(p\) to a period of positive volume, so the corollary follows.
\end{proof}


\begin{remark}
A more geometric argument can be developed to connect every pair of isoperiodic forms by a more explicit isoperiodic path. The argument uses a surgery -- Schiffer variations-- to connect first every form to a form with a single zero in its isoperiodic set. A few further Schiffer variations between forms with a single zero and an analysis of the action of the mapping class group, allows to prove that all isoperiodic forms with a single zero lie in the same connected component of the isoperiodic set. The advantage of this geometric argument with respect to the previous analytic argument,  is that it can be generalized to the case of genus one with any number of poles. This approach will be dealt with in \cite{ACD}. 
\end{remark}

\subsection{Genus two}
\label{ss:genus 2}
Recall that in the genus two case, there is a holomorphic embedding \( J : \mathcal{S}_{2,0} \rightarrow \mathfrak S_2\) of the Torelli space --the cover of \(\mathcal{M}_{g,0}\) whose covering group is the Torelli group \(\mathcal{I}(\Sigma_{g,0})\)-- in the Siegel space consisting of symmetric two by two matrices with complex coefficients and positive definite imaginary part. The map \(J\) associates to a homologically marked smooth curve of genus two its Jacobian matrix. The complement \( \mathfrak S_2 \setminus J (\mathcal{S}_{2,0}) \) is the set of Jacobian matrices of a nodal genus two homologically marked curve of compact type: this set is the union, parametrized by \( \text{Sp} (4,\Z) /\left(\text{Sp}(2,\Z) \times \text{Sp} (2,\Z)\right) \), of the images of the set of diagonal matrices with coefficients in the upper half-plane by an element of the group \( \text{Sp}(4,\Z)\).

\begin{theorem} \label{t: ramified covering genus two}
Given any \( p \in H^1 (\Sigma_2, \C/ \Z)\) of degree at least two, the restriction of the forgetful map  \[  \Omega^\pm\mathcal{S}_{2,\text{\sout{2}}}\rightarrow \mathcal{S}_{2,0}\text{ defined by } (C,x_-,x_+,m,\omega)\mapsto (C,m)\] to the isoperiodic set \( \per_2^{-1} ( p) \) is a ramified double cover over the open (connected) subset 
\[  \mathcal{S}_{2,0} \setminus \bigcup _{P \text{ lift of } p}   J^{-1} (\mathfrak S_P ).\]
The critical points correspond to homologically marked meromorphic \(1\)-forms of period \(p\) that are odd with respect to the hyperelliptic involution on the corresponding genus two curve. 
\end{theorem} 

Before giving the proof, let us recall the following classical lemma 

\begin{lemma}\label{l: Jacobian}
Let $C$ be a genus two smooth complex curve and $i:C\rightarrow C$ the hyperellptic involution. The map \( (q_1, q_2 ) \in C^2 \mapsto q_1 - q_2 \in J(C) \) is invariant by the involution \( I (q_1, q_2) = (i(q_2), i(q_1) ) \), and the induced map \( C^2 / I \rightarrow J(C) \) is the blow-up of \(J(C) \) at the point \(0\). The exceptional divisor is the quotient by \(I\) of the diagonal in \(C^2\).
\end{lemma}

\begin{proof}
Take two distinct couples of points of \(C\), \( (q_1, q_2) \neq (q_1', q_2') \), such that \( q_1 - q_2 = q_1' - q_2' \) in \( J(C)\). Then the divisor \( (q_1 + q_2') - ( q_1' + q_2) \) is principal, namely either \( \{ q_1, q_2'\} = \{ q_1' , q_2\} \), or there exists a non constant meromorphic function \( f : X\rightarrow \mathbb P^1 \) such that \(   (q_1 + q_2') - ( q_1' + q_2) = (f) \). The first case occurs when \( q_2= q_1\) and \( q_2' = q_1'\), corresponding to the level set of the trivial bundle in the Jacobian. In the second case, by uniqueness of the hyperelliptic involution on \(C\), the function \(f\) is the quotient of the hyperelliptic involution post composed by a biholomorphism \( C/ i \rightarrow \mathbb P^1\), and we then have \(\{ q_1', q_2 \} = i (\{ q_1, q_2'\})\), which shows that \( (q_1', q_2') = I (q_1, q_2)\).
\end{proof}

\begin{proof}[Proof of Theorem \ref{t: ramified covering genus two}]
Let \( (C,m)\in \mathcal{S}_{2,0}\) a homologically marked genus two curve. We denote by \( i_C : C\rightarrow C\) the hyperelliptic involution on \(C\), and by \( (L_p, \nabla _p)  \) the flat line bundle over \(C\) whose monodromy is given by the character \( \rho _p = \exp (2i \pi p)\circ m^{-1} \in H^1 (\Sigma _g , \C ^*) \).  Forgetting the flat connection, the line bundle \( L_p \) has degree zero, and defines an element of the Jacobian \( J(C) \). 

Given a meromorphic form \(\omega\) on \(C\) with two poles, both simple and having peripheral periods \(\pm 1\), the function \( \exp \left(2i\pi \int \omega \right) \) is a meromorphic function on \(\widetilde{C}\) which is \( \rho_p\)-equivariant. Hence, it induces a meromorphic section \( s_\omega : C\rightarrow L_p\), whose divisor is \(2i\pi\) times the residue divisor of \(\omega\). In particular, the bundle \( L_p\) is non trivial by Lemma \ref{l: Jacobian}. 

Reciprocally, suppose that \( L_p\) is non trivial. Then Lemma \ref{l: Jacobian} shows that it has two meromorphic sections, both having one simple zero and one simple pole. The logarithmic derivative of any such section with respect to \( \nabla _p\) gives the two desired meromorphic one forms on \(C\) with two poles, both simple and with peripheral periods \(\pm 1\), that have a period modulo \(\Z\) equal to \( p\).  

Notice that if we denote by \( (q_1,q_2)\) the two poles of the first section and by \( (q_1', q_2')\) the two poles of the second one, ordered so that the residues are \( 1/ 2i\pi\) at \(q_1\) and \(q_1'\) and \(-1/2i\pi \) at \(q_2\) and \(q_2'\), then \(  I( q_1,q_2) = (q_1', q_2')\). Hence the critical points of the projection of the isoperiodic set \( \text{Per}^{-1} (p)\) to the Torelli space \(\mathcal{S}_{2,0}\) correspond to the meromorphic differentials whose couple of ordered poles satisfy \( I(q_1, q_2) = (q_1, q_2)\). This means that the two poles are exchanged by the hyperelliptic involution.  Hence, if we denote by \(\omega\) such a meromorphic form, the form \( - i_C^* \omega\) has period \(p\) (this is because the hyperelliptic involution acts by multiplication by \(-1\) on the first cohomology group of \(C\)), and the same residue divisor as \(\omega\). Hence, the sum \( \omega + i_C^* \omega\) is a holomorphic one form with trivial period homomorphism: it needs to be identically zero, and thus \(\omega\) is odd with respect to the hyperelliptic involution.

To conclude, notice that if \(L_p\) is trivial, then it has a section that is nowhere vanishing. The logarithmic derivative of that section is a holomorphic one form on \(C\) whose period modulo \(\Z\) is equal to \(p\). Hence the period of that differential is a lift \(P\) of \(p\), and \(J(C)\) belongs to \(\mathcal S_P\). Reciprocally, each curve \(C\) whose Jacobian lies in \( \mathfrak S_P\) for some lift \(P\) of \(p\) has a holomorphic differential \(\omega\) whose period is \(P\). Then the multi-valued function \( \exp (2i\pi \int \omega) \) is \(\rho_p\)-equivariant and does not vanish, and defines a holomorphic section of \( L_p\) that vanishes nowhere. Hence the bundle \(L_p\) is trivial in this case.  
\end{proof}

\begin{corollary} \label{c: connectedness genus two}
If \(p \in H^1(\Sigma_2, \C/\Z) \) has degree at least three (it might be infinite), then the branched double cover of Theorem \ref{t: ramified covering genus two} has a branch point.  It is therefore connected. If \( p\) has degree two, there are no branch points. 
\end{corollary}
In next section we will determine an invariant that shows that the branched cover is disconnected when the degree is two  

\begin{proof}
Recall that a branch point of the map is a homologically marked meromorphic one form of period \(p\) which is odd with respect to the hyperelliptic involution on the corresponding genus two curve (see Theorem \ref{t: ramified covering genus two}). 

As is well known (see \cite{boissy} for details), there is a dictionary between meromorphic forms on a smooth genus $g$ curve and singular translation structures with some particular types of singularities on a compact topological surface of genus $g$. The period homomorphism of the form corresponds to the holonomy of the translation structure. The 'odd' point will be realized by producing a translation structure with the prescribed singularities, holonomy and  symmetry (corresponding to the hyperelliptic involution). One of the advantages is that the construction is not analytic, but geometric in spirit. The complex structure is obtained \textit{a posteriori}.

First assume that the image of \(p\) is not contained in the real line. Let \(P\) be a lift of \(p\) for which there exists a homologically marked holomorphic differential \( ( C, m , \omega) \) with period \( (\int \omega) \circ m= P\), whose existence has been established in  Proposition \ref{p: lifts}. The form \(\omega\) is odd with respect to the hyperelliptic involution, namely it satisfies \( i_C^* \omega = -\omega\). Among the six points fixed by the hyperelliptic involution on \(C\), there are at least four at which the form \(\omega\) does not vanish. At one of these points, slit \(C\) along a small vertical segment of length \(l\) invariant under the hyperelliptic involution, slit the infinite cylinder \( \C /\Z \) along a vertical segment of length \(l\), and glue these two slit surfaces together in order to get a homologically marked meromorphic one form which is odd with respect to the hyperelliptic involution. So we are done in this case. 

Now let us assume that \(p\) takes only real values and that its degree is at least three.  We begin by the following very elementary  

 \vspace{0.2cm} 
 
 \noindent \textit{Claim 1: there exists a symplectic basis \( a_1, b_1, a_2, b_2\) of \( H_ 1 (\Sigma_2, \Z)\) such that none of the periods \( p( a_1 ) , p(a_2) \) nor \( p(a_1+a_2) \) vanish.}
 
 \vspace{0.2cm} 
 
 
 \begin{proof} 
 Since \(p\) is not zero modulo \(\frac{1}{2} \mathbb Z\), by Corollary \ref{c:padmissible decompositions exist}, there exists a symplectic decomposition \(H_1(\Sigma_2,\mathbb Z) = H_1\oplus H_2 \) so that the restriction of \(p\) to each module \(H_k\), \(k=1,2\), is not zero modulo \(\frac{1}{2} \mathbb Z\). In each of the symplectic submodules \(H_k\), let \(a_k , b_k\) be a symplectic basis so that \( p(b_k)=0 \). Then, \(p(a_k)\) does not belong to \( \frac{1}{2}\mathbb Z /\mathbb Z\) by construction. If \( p (a_1 + a_2) \neq  0\), we are done. Otherwise, \(p(a_2) = - p(a_1)\), so define \(a_2 ' = -a_2\) and \(b_2' = - b_2\). We then have \( p (a_1 + a_2 ') = 2p (a_1) \neq 0\), and the basis \( a_1, b_1 , a_2 ', b_2' \) works.   
\end{proof} 
 
 Now let us denote by \(p_-, p_+\) two distinct points in \( \Sigma_2\), and let \( \widetilde{a_1},\widetilde{b_1}, \widetilde{a_2}, \widetilde{b_2}\in H_1(\Sigma_2 \setminus \{p_-,p_+\}, \Z)\) lifts of the elements \( a_1,b_1,a_2,b_2\). Denoting by \( \pi^{\pm} \in H_1(\Sigma_2\setminus \{p_-, p_+\},\Z ) \) the peripheral homology classes  around \(p_\pm\), the family \( \widetilde{a_1}, \widetilde{b_1}, \widetilde{a_2}, \widetilde{b_2}, \pi^+\) is a basis of \( H_1 (\Sigma_2\setminus \{p_-,p_+\}, \Z)\). We have the following 
 
 \vspace{0.2cm} 
 
 \textit{Claim 2: one can choose the symplectic basis \(a_1, b_1, a_2, b_2\) and its lift and \( P \in H^1 (\Sigma_2 \setminus \{p_-, p_+\}, \R) \) whose reduction modulo \(\Z\) equals \(p\), and which is such that \[ P (\widetilde{a_1} ), P(\widetilde{a_2} )  >0\text{ and } P(\widetilde{a_1} ) + P(\widetilde{a_2}) <1.\]}
 \begin{proof} Consider the basis \( a_1, b_1, a_2, b_2\) constructed in the Claim 1. Let \(P\) be any lift of \(p\) satisfying that \( P( \widetilde{a_k}) \) is the determination of \( p(a_k) \) modulo \(\Z\) that belongs to \( (0, 1) \).  Observe that \( P (\widetilde{a_1} ) + P (\widetilde{a_2}) \) belongs to \( (0, 2) \setminus \{1\}\) since \(p(a_1)+p(a_2) \neq 0\). If \( P(\widetilde{a_1}) + P(\widetilde{a_2}) <1\) we are done. Otherwise, we have \(P(\widetilde{a_1}) +P(\widetilde{a_2}) >1\). In this case, introduce the new symplectic basis \(a_k'= -a_k, b_k'=-b_k\), and its lift \( \widetilde{a_k'}=-\widetilde{a_k}\), \(\widetilde{b_k'}=-\widetilde{b_k}\). Define \(P'\) so that \( P' (\widetilde{a_k'})\) is the determination of \( p(a_k ') \) modulo \(\Z\) lying in the interval \( (0, 1)\). We have \( P' (\widetilde{a_k'}) = 1- P( \widetilde{a_k}) \). Hence, \(P' (\widetilde{a_1'})+P' (\widetilde{a_2'}) = 2 - (P(\widetilde{a_1})+P(\widetilde{a_2})) <1\), and we are done.  \end{proof}

Let now \( \widetilde{a_3}\in H_1 (\Sigma_2\setminus \{p_-, p_+\}, \Z)\) be defined by the relation \( \widetilde{a_1}+\widetilde{a_2}+\widetilde{a_3} = \pi^+\). Since \( P (\pi^+)=1\), we have 
\[ P(\widetilde{a_k})>0 \text{ for } k=1,2,3 \text{ and } P(\widetilde{a_1}) + P(\widetilde{a_2} ) + P(\widetilde{a_3}) =1 .\]    

Consider a translation surface \( X^+\) made of four vertical cylinders \( C_0, C_1, C_2, C_3\). The first one, \(C_0\), is a semi-infinite cylinder of circumference \(1\) isomorphic to \( \overline{\mathbb H}/\Z\), the other three, \(C_1, C_2, C_3\), are compact cylinders of height \(1\) and of circumference \( P(\widetilde{a_1}) , P(\widetilde{a_2})\), and \(P(\widetilde{a_3})\). We glue these three cylinders below \(C_0\), to form a translation surface homeomorphic to a sphere minus three open discs and one puncture. 

\begin{figure}[h]
     \centering
     \includegraphics[height=7cm]{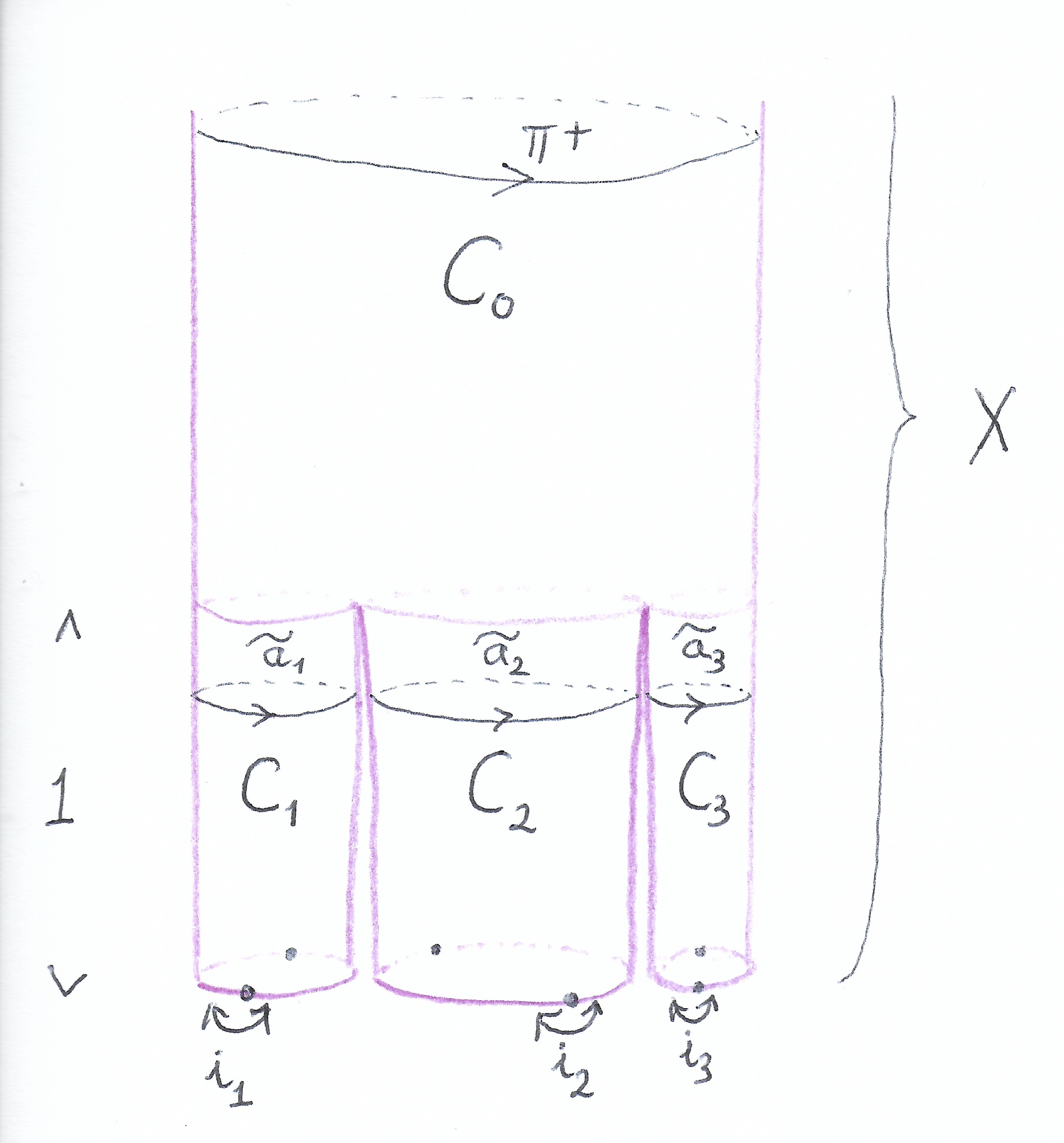}
     \caption{The surface with boundary \(X^+\)}
     \label{fig:branch_point}
 \end{figure}

This translation surface is non compact and has three boundary components \(\partial _k  X\) that are horizontal geodesics of length \( P(\widetilde{a_k}) \). Choosing reversing orientation isometries \(i_k \) of the boundary components \( \partial _k X^+\), one construct a translation surface \(X\) in the following way: let \( X^-\) be the translation surface obtained from \(X^+\) by modifying each charts of the translation structure by its composition with \(-\text{id}\). The reversing orientation isometries \(i_k\) can be viewed as isometries from \(\partial _k X^+ \subset \partial X^+\) to \(\partial _k X^-\subset \partial X^-\). Use these isometries to glue \( X^+\) to \(X^-\) and set \(X\) to be the resulting translation surface. The identifying map from \( X^+\) to \(X^-\) and from \( X^- \) to \( X^+\) defines an involution \(i_X: X\rightarrow X\) which in the charts of the translation structure of \(X\) takes the form \( z\mapsto - z + \text{cst}\).  Let \( m: H_1 (\Sigma_g\setminus \{q^+, q^-\}, \Z) \rightarrow H_1(X, \Z) \)  be any marking of \( X\) sending \( \pi ^\pm \) to the peripheral of \( C_0^\pm\), and \( \widetilde{a_k}\) to the boundary \(\partial_k X^+\). If we denote by \(\omega\) the differential on \(X\) defined by \(\omega = dz \) in any translation chart \(z\), we then have \(\int _{m(\widetilde{a_k}) } \omega = P(\widetilde{a_k})\). 

This construction depends on some continuous parameters that are the isometries \( i_ k\). Conjugating \( i_k \) by a rotation of angle \(\theta_k\) results in twisting with  angle \(\theta_k\)  along the geodesic \(\partial _k X^+\). The periods of the new homologically marked holomorphic form \( (X', \omega', m') \) (the marking \(m'\) is obtained from the marking \(m\) using the Gauss Manin connection) can be computed by the following simple formula
\[ \int _{m(\widetilde{\gamma})} \omega ' = \int _{m(\widetilde{\gamma})} \omega + \sum _l  \theta_l\  \widetilde{a_l}\cdot \widetilde{ \gamma } ,\]
where as usual the symplectic product is denoted by a dot. We can then choose the angles \(\theta_k\), for \(k=1,2\), and \(\theta_3=0\), in such a way that 
\[ \int _{m(\widetilde{b_k})} \omega' = P( \widetilde{b_k}) .\]
The form  \( (\overline{X'}, \omega', \overline{m'})\in \Torelliformg \) obtained from \( (X, \omega, m) \) by adding two points at its two ends is the desired element of  \( \text{Per}^{-1}( p) \) which is odd with respect to hyperelliptic involution. Hence, \(\text{Per}^{-1} (p)\) is connected in this case. 

To conclude we need to show that if the degree of \(p\) is two there are no branch points of the branched double cover.  It suffices to establish that there is no meromorphic differential with two simple poles on a genus two curve which is at the same time of degree two and odd with respect to the hyperelliptic involution. Now, given a meromorphic form \( \omega\) of degree two, the function \( \exp \left(2i\pi \int \omega\right) \) is the unique hyperelliptic function up to post-composition, hence \(\omega \) is even with respect to the hyperelliptic involution in this case, not odd.  

\end{proof}

\section{Proof of Theorem \ref{t:disconnected fibers}}

Given a finite subset \( E\subset \Pone \), the homology group \( H_1 (\Pone \setminus E, \Z/2\Z) \) is naturally isomorphic to \( M_E = \frac{\bigoplus _{e\in E} \Z/2\Z e }{ \Z/2\Z \sum _{e\in E} e }\). In the sequel, we will consider subsets of the Riemann sphere that contain the points \( 0\) and \(\infty\) of cardinality \( 2g+2\). We will fix once for all such a subset, that we denote \(E_0\). Observe that any bijection \(\sigma\) from \(E\) to \( E_0\) that map the points \(0\) and \(\infty\) to themselves, induces a linear map \(\varphi_{\sigma} : M_E\rightarrow M_{E_0}\). Two such linear maps differ by post-composition by a map of the form \(\varphi _{\sigma}\), where \( \sigma\) is a permutation of \( E_0\) that fixes \(0\) and \(\infty\). Given a \(\mathbb Z/2\mathbb Z\)-module \(M\), we say that two morphisms \( f: M \rightarrow M_{E_0}\) are equivalent if they differ by post-composition by a morphism of the form \( \varphi_{\sigma}\) where \( \sigma\) is a permutation of \(E_0\) that fixes the points \(0\) and \(\infty\). The set of equivalence classes is denoted \( S_{2g} \backslash \text{Hom} (M, M_{E_0})\). 

\begin{definition}
 Given a marked meromorphic form $(C,m,\omega)$ of degree two on a smooth curve we define its Arnold invariant \(\text{Ar}= \text{Ar}(C,m,\omega)\) as the class \([A]\) in \( S_{2g} \backslash \text{Hom} (H_1(\Sigma_g, \mathbb Z/2\mathbb Z), M_{E_0})\) of the homomorphism \(A= \varphi_\sigma \circ f_* \circ m\), where  $$f_*:H_1(C,\mathbb{Z}/2\mathbb{Z})\rightarrow H_1 (\Pone \setminus\text{VC}(f),\mathbb{Z}/2\mathbb{Z})$$ 
 is the well defined homomorphism induced in the $\mathbb{Z}/2\mathbb{Z}$-homology of $C$ by the restriction of the branched double covering $f=\exp(2\pi i\int\omega):C\rightarrow \Pone$ to the complement of the ($2g+2$ simple) branch points (see \cite{Arnold}). In this formula, \( \text{VC} (f)\) is the  set of critical values of the covering \(f\), 
and \( \sigma\) is a bijection from \( \text{VC} (f) \) to \(E_0\) mapping \(0\) to \(0\) and \(\infty\) to \(\infty\).
\end{definition} 

For each degree two homomorphism \(p\in H^1(\Sigma_g,\mathbb{C}/\mathbb{Z})\) the Arnold invariant is \textit{locally} constant on $\per^{-1}(p)$. 
On the other hand, the fact that $\ker(p)$ has rank $2g-1$, implies that there is a non-trivial stabilizer $\text{Stab}(p)$ of $p$ in $\text{Sp}(H_1(\Sigma_g,\mathbb{Z}))$. 
This stabilizer acts on $\per^{-1}(p)$ by pre-composition on the marking. The idea of the proof of Theorem \ref{t:disconnected fibers} is to show that in the orbit of a point in $\per^{-1}(p)$ there are at least two values of the Arnold invariant, and therefore $\per^{-1}(p)$ has at least as many components as the number of values that the Arnold invariant takes. 

Note that we can recover the period \( p\in H^1 (\Sigma_g, \frac{1}{2} \Z / \Z)\simeq H^1 (\Sigma_g, \Z /2\Z )\) from the Arnold's invariant of a marked form \( (C,m, \omega) \) of degree two. Indeed, the group \( M_{\{0,\infty\}} \) is naturally isomorphic to \( \Z/ 2\Z\) (such a module have no non trivial automorphisms); the period \(p\)  is just the composition of a representative \(A\) of \( \text{Ar}(C,m,\omega)\) with the projection \( M_{E_0}\rightarrow M_{\{0,\infty\}}\simeq M_{E_0} / \bigoplus_{e\in E_0\setminus \{0,\infty\}} \Z/2\Z e\). In particular, the stabilizer of the Arnold's invariant of an element \( (C,m,\omega)\) of degree two in the symplectic group \(\text{Sp} (H_1(\Sigma_g, \Z/2\Z))\) is contained in the stabilizer of the period \(p=\text{Per}(C,m,\omega)\).  

\begin{lemma}\label{l:counting finite groups}
For any $g\geq 2$, given any \( (C,m,\omega)\in\Torelliformg\) of degree two,  the stabilizer in the group $\text{Sp}(H_1(\Sigma_g,\mathbb{Z}/2\mathbb{Z}))$ of the reduction modulo two of the period homomorphism \( p= \text{Per} (C,m,\omega)\)   is larger than the stabilizer of the Arnold class \(\text{Ar} (C,m,\omega)\).
\end{lemma}

\begin{proof}
Since \(p\neq 0\) mod two, the stabilizer of the reduction modulo two of \(p\) in  \(\text{Sp} (H_1(\Sigma_g, \Z/2\Z))\) is in bijection of the set of symplectic basis of \(H^1 (\Sigma_g, \Z/2\Z) \) that contain the element \(p\) as the first member. The number of such basis is the product
\[ 2^{2g-1} (2^{2g-2} - 1) 2^{2g-3} \times \ldots \times (2^2- 1 ) \times 2 . \]

Let us now bound from below the number of elements in the stabilizer of the Arnold's class \( \text{Ar} = \text{Ar} (C, m, \omega) \), representated by a map \( A : H_1 (\Sigma_g, \Z/2\Z) \rightarrow M_{E_0}\). 

Observe that the map \( A\) is injective, and that its image is the subset \( M_{E_0}^{even} \subset M_{E_0}\) consisting of formal sums \( \sum _{e\in F} e\), where \(F\) is a subset of \(E_0\) with an even number of elements.

Given any element \( M\in \text{Sp} (2g, \Z/2\Z) \) which fixes \( \text{Ar} \), there exists a permutation \(\sigma \) of the set \(E\) such that \begin{equation} \label{eq: equivariance} A\circ M^{-1} = \varphi_\sigma \circ A .\end{equation}
Since \(M_{E_0}\) has more than two elements, the action of \( \varphi_{\sigma}\) on \(M_{E_0}^{even}\) completely determines \(\sigma\), the permutation \( \sigma\) satisfying \eqref{eq: equivariance} is well-defined, and the induced map  
\begin{equation} \label{eq: morphism} M\in \text{Stab} _{\text{Sp} (H_1(\Sigma_g, \Z/2\Z) } (\text{Ar})\mapsto \sigma \in  S_{2g} \end{equation}
is a homomorphism. Now, because the map \(A\) is injective, the morphism \eqref{eq: morphism} is injective. We thus get that stabilizer of \( \text{Ar}\) in \(\text{Sp} (H_1(\Sigma_g, \Z/2\Z))\) has at most \((2g) !\) elements. The lemma follows.
\end{proof}

\begin{corollary}\label{c: several Arnold invariants}
If $g\geq 2$ and $(C,m,\omega)\in\Torelliformg$ has degree two period homomorphism $p=\per(C,m,\omega)$,  there exist at least two values of the Arnold invariant in the orbit $$\text{Stab}(p)\cdot(C,m,\omega)\subset \per^{-1}(p).$$  
\end{corollary}

\begin{proof}
The reduction modulo two map \(\text{Sp} (H_1 (\Sigma_g, \Z) )\rightarrow \text{Sp} (H_1 (\Sigma_g, \Z/2\Z) ) \) is  surjective. From Lemma \ref{l:counting finite groups} we deduce that there exist two elements in the orbit that are not equivalent by post-composition by an element in $\text{Sym}_{2g}$. 
\end{proof}

Theorem \ref{t:disconnected fibers} follows from Corollary \ref{c: several Arnold invariants}. The latter also implies that the branched double cover of Theorem \ref{t: ramified covering genus two} is not only unbranched,  as stated in Corollary \ref{c: connectedness genus two}, but also disconnected.

\section{Extension of the period map and bordification}
\label{s:extension of the period map and bordification}

\subsection{Augmented Teichm\" uller space}
\label{ss:augmented teich definitions}
As in  \cite{CDF2}, we need to extend the domain of definition of the period map in to appropriate bordifications of the space $\Torelliformg$. For the sake of completeness we have included a general theory of the period map for moduli spaces of forms with simple poles and their extension to the augmented Torelli spaces in the Appendix. 

In this section we give a brief outline and focus on some particular features of the case of two poles. 

The moduli space of genus $g$ smooth curves with $n$ marked points $\mathcal{M}_{g,n}$ admits a Deligne-Mumford-Knudsen compactification $\overline{\mathcal{M}}_{g,n}$ as a complex orbifold (see \cite{DM,K1,K2,KM}). Each point in the boundary $\partial\mathcal{M}_{g,n}=\overline{\mathcal{M}}_{g,n}\setminus \mathcal{M}_{g,n}$  corresponds to a class under conformal isomorphism of the following geometric object:
\begin{definition}
  A connected complex curve $C$ with $n$ marked distinct points $q_1,\ldots,q_n\in C$ is said to be stable if its singularities are nodes that do not coincide with any of the marked points,  and the closure $C_i$ of each component of $C^*:=C\setminus\text{Sing}(C)$, called a part of $C$, has a group of automorphisms that fix the marked points and the boundary points that is finite. The normalization of $C$ is the (possibly disconnected) smooth curve $\hat{C}=\sqcup C_i$ marked by the points $q_1,\ldots, q_n$ and the pairs of points corresponding to the nodes of $C$. A stable curve $C$ is said of compact type if every node separates $C$ in two components. Otherwise $C$ is said to be of non-compact type.
\end{definition}

The arithmetic genus of a stable curve is $g=h^1(C,\mathcal{O})$. When $C$ has $\delta$ nodes and its normalization has $\nu$ components of genera $g_1,\ldots, g_{\nu}$, the arithmetic genus satisfies (see \cite{Harris}[p. 48]) $$g=\sum_{i=1}^\nu (g_i-1)+\delta+1$$

The boundary $\partial\mathcal{M}_{g,n}$ is a normal crossing divisor and, as such, is stratified. Its regular points correspond to the subset of curves with a single node. 

There are natural attaching maps \begin{equation}
\label{eq:attach1}    \overline{\mathcal{M}}_{g_1,n_1\cup *}\times\overline{\mathcal{M}}_{g_2,n_2\cup\star}\rightarrow \overline{\mathcal{M}}_{n_1+n_2}
\end{equation} and 
\begin{equation}\label{eq:attach2}
\mathcal{M}_{g,n\cup\{*,\star\}}\rightarrow \overline{\mathcal{M}}_{g+1,n}
\end{equation}
where the special points are identified into a new node. Their image is therefore in the boundary.  

There are also forgetful maps (forget last marked point with stabilization of the resulting curve, i.e. collapse non-stable components) 
\begin{equation}
    \overline{\mathcal{M}}_{g,n\cup{*}}\rightarrow\overline{\mathcal{M}}_{g,n} 
\end{equation}
All of the said maps are holomorphic. By convention we omit the zero subindex when $n=0$ and write $\mathcal{M}_g$ and $\overline{\mathcal{M}}_g$ instead. 

Recall from the introduction that the universal cover  of $\mathcal{M}_{g,n}$ can be identified with the Teichm\" uller space $\mathcal{T}_{g,n}$ and its covering group with the mapping class group $\text{Mod}(\Sigma_{g,n})$ for a reference surface  $\Sigma_{g,n}=(\Sigma_g,p_1,\ldots,p_n)$ of genus $g$ with $n$ ordered marked points that we will sometimes denote as $P=(p_1,\ldots,p_n)$. When $n=0$ we omit the subindex and write $\Sigma_g=\Sigma_{g,0}$. 

The Teichm\"uller space can be bordified to the so called augmented Teichm\" uller space $\overline{\mathcal{T}}_{g,n}$, a stratified topological space that is no longer a manifold but where the action of $\text{Mod}(\Sigma_{g,n})$ extends naturally with quotient $\overline{\mathcal{T}}_{g,n}/\text{Mod}(\Sigma_{g,n})$ homeomorphic to $\overline{\mathcal{M}}_{g,n}$ (see \cite{ACG}). The quotient map $$\overline{\mathcal{T}}_{g,n}\rightarrow \overline{\mathcal{T}}_{g,n}/\text{Mod}(\Sigma_{g,n})$$ is a branched cover, branching over the boundary points. 
Let us recall its definition and some properties.
\begin{definition}
 A homotopical marking (or sometimes a collapse) of a connected genus $g$ stable curve $C$ with $n$ ordered pairwise distinct ordered marked points $Q=(q_1,\ldots, q_n)\in (C^*)^n$ is a continuous surjection $f:\Sigma_{g,n}\rightarrow (C, q_1,\ldots,q_n)$ such that $f(p_i)=q_i$, the preimage of each node is a simple closed curve on $\Sigma_g\setminus P$ and on each component of $\Sigma_g\setminus f^{-1}(N)$ where $N$ is the set of nodes, the map $f$ is a homeomorphism onto a part of $C$ that preserves the orientation. 
\end{definition}

\begin{definition}
 A homotopically marked stable curve with $n$ marked points is a marked stable curve $(C, q_1,\ldots,q_n)$ together with a homotopical marking $f:(\Sigma_{g},P)\rightarrow (C,Q)$. Two homotopically marked stable curves $f_i:(\Sigma_g,P)\rightarrow (C_i,Q_i)$ for $i=1,2$ are said to be equivalent if there exists a conformal isomorphism $g:C_1\rightarrow C_2$ such that $g\circ f_1$ is homotopic to $f_2$ relative to $P$. The class of a $\Sigma_{g,n}$ marked stable curve will be denoted by $\{f:\Sigma_{g,n}\rightarrow (C, q_1,\ldots, q_n)\}$ or by $(C,Q,\{f\})$. 
\end{definition}

When $C$ is non-singular the homotopy class of a collapse corresponds to a unique isotopy class of homeomorphisms. 

\begin{remark}\label{rem:homotopically deform to dehn}
If $\Delta:\Sigma_{g,n}\rightarrow\Sigma_{g,n}$ is a Dehn twist around a simple closed curve  in $\Sigma_g$ that is collapsed by the marking $f:(\Sigma_g,P)\rightarrow (C,Q)$ to a point, then $(C,Q)$ marked by $f\circ\Delta$ is equivalent to the same curve marked by $f_1$. 
\end{remark}

\begin{definition}
The augmented Teichm\" uller space $\overline{\mathcal{T}}_{g,n}$ is the set of all homotopically marked stable genus $g$ curves with $n$ marked points  up to equivalence. It contains the Teichm\" uller space $\mathcal{T}_{g,n}$ of isotopy classes of marked \textit{smooth} curves, and the boundary $\partial\mathcal{T}_{g,n}=\overline{\mathcal{T}}_{g,n}\setminus \mathcal{T}_{g,n}$. 
\end{definition}
The topology of Teichm\" uller space extends to a topology on Augmented Teichm\" uller space in such a way that around a boundary point it is not locally compact. (see \cite{ACG} and references therein). There is a natural stratification of the boundary by complex manifolds. Each stratum is defined by the number of nodes of the curves and that number gives precisely the codimension of the manifold with respect to the dimension of the open stratum $\mathcal{T}_{g,n}$.  

The local structure of the stratification is described in \cite{ACG}. It is  is a local abelian ramified cover of a normal crossing divisor (we borrow the terminology from in  Section 4 of \cite{CDF2}). This property allows to mimic the description of the dual complex associated to a normal crossing divisor as follows: a vertex is associated to each connected component of the codimension one stratum. Each connected component of the stratum of codimension $k$ lies in the closure of precisely $k$ connected components of the codimension one stratum.  We attach a  $k$-simplex   to $k$ given vertices for each connected component of the intersection of the closures of the $k$ corresponding components of strata.  The resulting complex is denoted $\mathcal{C}(\overline{\mathcal{T}}_{g,n})$. It is isomorphic to the curve complex $\mathcal{C}_{g,n}$ defined on a genus $g$ closed surface with $n$ marked points. (see \cite{CDF2}).

\subsection{Homology quotients}

In this paper we will need to work on some intermediate quotients of $\overline{\mathcal{T}}_{g,2}$. 
 associated to particular subgroups of $\text{Mod}(\Sigma_{g,2})$ characterized by their (trivial) action on certain homology groups with integral coefficients. 
 
 The pointed Torelli group $\mathcal{I}(\Sigma_{g,n^*})$ defined by the exact sequence \begin{equation}1\rightarrow \mathcal{I}(\Sigma_{g,n^*})\rightarrow \text{Mod}(\Sigma_{g,n})\rightarrow \text{Aut}\big(H_1(\Sigma_g\setminus \{p_1,\ldots,p_n\})\big)\end{equation}  
and the classical Torelli group represented in $\text{Mod}(\Sigma_{g,n})$ by the exact sequence
\begin{equation}1\rightarrow \mathcal{I}(\Sigma_{g,\text{\sout{n}}})\rightarrow \text{Mod}(\Sigma_{g,n})\rightarrow \text{Aut}\big(H_1(\Sigma_g)\big).\end{equation} 

Both groups are related. Denote $\Pi_n$ the rank $n-1$ submodule generated by the $n$ peripheral curves around the punctures in $H_1(\Sigma_g\setminus\{p_1,\ldots,p_n\})$ we have an exact sequence \begin{equation}\label{eq:peripheral ex seq}0\rightarrow\Pi_n\rightarrow H_1(\Sigma_g\setminus\{p_1,\ldots,p_n\}) 
\rightarrow \frac{H_1(\Sigma_g\setminus\{q_1,\ldots,q_n\})}{\Pi_n}\cong H_1(\Sigma_g)\rightarrow 0\end{equation}

The group $\text{Mod}(\Sigma_{g,n})$ acts on \eqref{eq:peripheral ex seq} preserving it, since it fixes every class in $\Pi_n$. 
In particular, $\mathcal{I}(\Sigma_{g,n^*})\subset\mathcal{I}(\Sigma_{g,\text{\sout{n}}})$ is a subgroup. We claim that there is an isormophism $$\frac{\mathcal{I}(\Sigma_{g,\text{\sout{n}}})}{\mathcal{I}(\Sigma_{g,n^*})}\cong\text{Hom}(H_1(\Sigma_g),\Pi_n)$$
Indeed, take $\phi\in\mathcal{I}(\Sigma_{g,\text{\sout{n}}})$. The endomorphism defined by $\phi_*-\text{Id}$ on  $H_1(\Sigma_g\setminus\{p_1,\ldots,p_n\})$ has image in $\Pi_n$ and $\Pi_n$ in its kernel. Thanks to the isomorphism in \eqref{eq:peripheral ex seq} it defines an element in $\text{Hom}(H_1(\Sigma_g),\Pi_n)$.  The kernel of the group homomorphism $\phi\mapsto \phi_*-\text{Id}$ is precisely $\mathcal{I}(\Sigma_{g,n^*})$. 

The previous analysis implies that the quotient map \begin{equation}
    \label{eq:special cover}
\frac{\mathcal{T}_{g,n}}{\mathcal{I}(\Sigma_{g,n^*})}\rightarrow \frac{\mathcal{T}_{g,n}}{\mathcal{I}(\Sigma_{g,\text{\sout{n}}})}\end{equation} has covering group $\text{Hom}(H_1(\Sigma_g),\Pi_n)$.

Define $$\overline{\mathcal{S}}_{g,n^*}=\overline{\mathcal{S}}(\Sigma_{g,n^*})=\frac{\overline{\mathcal{T}}_{g,n}}{\mathcal{I}(\Sigma_{g,n^*})}\text{ and }\overline{\mathcal{S}}_{g,\text{\sout{n}}}=\frac{\overline{\mathcal{T}}_{g,n}}{\mathcal{I}(\Sigma_{g,\text{\sout{n}}})}$$ Then the quotient map \eqref{eq:special cover} extends to a map \begin{equation}\label{eq:extension of cover} \overline{\mathcal{S}}_{g,n^*}\rightarrow \overline{\mathcal{S}}_{g,\text{\sout{n}}}\end{equation}
that has branch points possibly at points with non-separating nodes. An example of branched point with $(g,n)=(4,2)$ is depicted in Figure \ref{fig:nodal branch point}.

\begin{figure}[h]
     \centering
     \includegraphics[height=4cm]{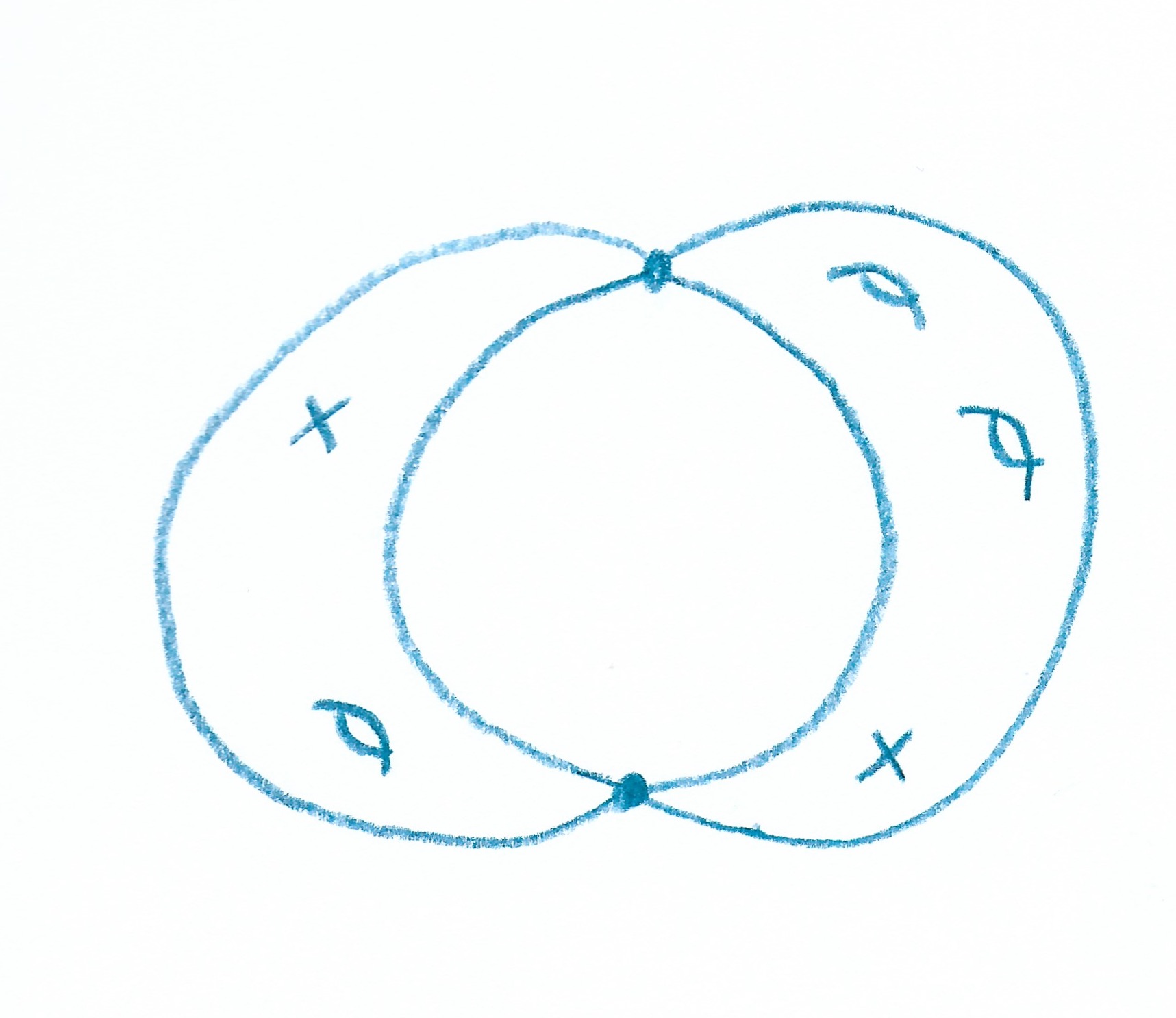}
     \caption{An example of branch point of the extension \eqref{eq:extension of cover} with $(g,n)=(4,2)$. The Dehn twist around the two curves of the curve system acts trivially on $\overline{\mathcal{S}}_{4,\text{\sout{2}}}$ and the depicted boundary point of $\mathcal{S}_{4,2^*}$ but non-trivially in the neighbouring points outside the boundary }
     \label{fig:nodal branch point}
 \end{figure}
 
 \subsection{Homological characterization of strata of curves of compact type}
The stratification by the number of nodes of $\overline{\mathcal{T}}_{g,n}$ is invariant under the action of the group $\mathcal{I}(\Sigma_{g,n^*})$ (resp.  $\mathcal{I}(\Sigma_{g,\text{\sout{n}}})$ ) and induces a stratification of  $\overline{\mathcal{S}}_{g,n^*}$ (resp.  $\overline{\mathcal{S}}_{g,\text{\sout{n}}}$). Both stratifications are still locally abelian ramified covers of a normal crossing divisor. Therefore we can define the dual boundary complex $\mathcal{C}(\overline{\mathcal{S}}_{g,n^*})$ ( resp.  $\mathcal{C}(\overline{\mathcal{S}}_{g,\text{\sout{n}}})$ ) which is isomorphic to the quotient of the curve complex $\mathcal{C}_{g,n}/ \mathcal{I}(\Sigma_{g,n^*})$ (resp. $\mathcal{C}_{g,n}/ \mathcal{I}(\Sigma_{g,\text{\sout{n}}})$).

To every class of point $(C,Q,\{f\}) \in\overline{\mathcal{S}}_{g,\text{\sout{n}}}$ there corresponds a surjective homomorphism $$m=[f]_*:H_1(\Sigma_g,\mathbb{Z})\rightarrow H_1(C,\mathbb{Z}).$$

\begin{lemma}\label{l:homological characterization}
If $(C,Q,\{f\})$ and $(C,Q,\{f'\})$ satisfy 
\begin{itemize}
    \item $[f]_*=[f']_*$
    \item $\text{rank}(\ker [f]_*)\leq 1$,
\end{itemize}
then both points define the same class in $\overline{\mathcal{S}}_{g,\text{\sout{n}}}$.
\end{lemma}
\begin{proof}
There exists $\phi\in\text{Mod}(\Sigma_{g,n})$ such that $f'=f\circ \phi$. By hypothesis $[f\circ\phi]_*=[f']_*=[f]_*$. The endomorphism of $H_1(\Sigma_g)$ defined by  $\phi_*-\text{Id}$ has its image in $\ker[f]_*$. If $\ker [f]_*=0$  we are done. 

If $\ker[f]_*$ has rank one, it is a primitive module, generated by the class $[c]$ of a non-separating simple closed curve $c$ and $$\phi-\text{Id}= [c]\varphi \quad \text{where }   \varphi\in H^1(\Sigma_g,\mathbb{Z})$$
By duality $\varphi$ is the symplectic dual of some element of $\ker[f]_*\subset H_1(\Sigma_g)$, hence there exist $k\in\mathbb{Z}$ such that  $$\phi(a)=a+k (a\cdot c)[c]\quad\forall a\in H_1(\Sigma_g).$$ 
This homomorphism is precisely the action in homology induced by the composition $\Delta_c^{\circ k}$ of $k$ times the Dehn twist around the curve $c$. By Remark \ref{rem:homotopically deform to dehn},  $f''=(f\circ\phi)\circ(\Delta_c^{-1})^{\circ k}$ is an equivalent topological marking to $f'$ that falls in the initial case. 
\end{proof}
In particular, in $\overline{\mathcal{S}}_{g,\stkout{2}}$ every point over a curve of compact type $C$ is characterized by a tuple $(C,x_-,x_+,m)$  where 
$m:H_1(\Sigma_g,\mathbb{Z})\rightarrow H_1(C,\mathbb{Z})$ is a surjective homomorphism.
There is a natural direct sum decomposition into symplectic submodules  $$H_1(C,\mathbb{Z})=H_1(C_1,\mathbb{Z})\oplus\cdots\oplus H_1(C_k,\mathbb{Z})$$ where $C_1,\ldots,C_k$ are the parts of $C$, each of genus $g_i=\text{genus}(C_i)$. This decomposition induces, via the marking $m$,  a decomposition \begin{equation}\label{eq:compact type decomposition}H_1(\Sigma_g,\mathbb{Z})=W_1\oplus\cdots\oplus W_k\end{equation} into pairwise orthogonal symplectic submodules of ranks $\text{rank}(W_i)=2g_i$.


As proven in \cite{CDF2}, using the homological characterization of Lemma \ref{l:homological characterization}, the restriction of attaching maps as in equation \eqref{eq:attach1} can be defined at the level of the coverings  
\begin{equation}
\label{eq:attaching at torelli level}\mrs_{g_1,\stkout{n_1\cup*}}\times\mrs_{g_2,\stkout{n_2\cup\star}}\rightarrow\overline{\mathcal{S}}_{g_1+g_2,\stkout{n_1+n_2}}\end{equation}  by the relation \begin{equation}\label{eq:attach torelli cpt type}(C_1,Q_1\cup*,m_1)\times(C_2,Q_2\cup\star,m_2)\mapsto (C_1\vee_{*=\star} C_2,Q_1\cup Q_2, m_1\oplus m_2).\end{equation} 
All points in the image have the same associated symplectic decomposition. Moreover, if $n_1+n_2=2$ then the image is contained in $\overline{\mathcal{S}}_{g,\text{\sout{2}}}$ and the range is a product of spaces $\mathcal{S}_{g,\stkout{n}}$ where $n=0,1$ or $2$. 

\begin{remark}\label{rem:extension to pts with one ns node}
Using Lemma \ref{l:homological characterization} we can extend the domain of the map \eqref{eq:attaching at torelli level} to pairs of curves with marked points such that at at most one has one non-separating node
\end{remark}


\subsection{Stable meromorphic forms with two simple poles}
\begin{definition}
 A stable meromorphic form with two simple poles and residues $\pm 1$ on a stable curve with two marked points $(C,x_-,x_+)\in\overline{\mathcal{M}}_{g,2}$ is a meromorphic form on each component of $C^*$ having at worst simple poles at the points corresponding to the nodes, such that the sum of residues at the branches of $C$ at any point $x\neq x_\pm$ is zero. Moreover, the residue is $-1$ at $x_-$ and $+1$ at $x_+$. The space of all such forms on a given stable curve $(C,x_-,x_+)$ is denoted by $\Omega^{\pm}(C,x_-,x_+)$. It is an affine space directed by the $g$-dimensional vector space $\Omega(C)$ of stable (holomorphic) forms on $C$ (with zero residue sum at all points).  
  \end{definition}
  On any component of $C^*$ the residue theorem holds for the restricted form,  telling us that the sum of residues in the component is zero. To be able to integrate $\omega$ along a path in $C$ it needs to avoid all poles of the restrictions to the components of $C^*$. 

In contrast with meromorphic forms on smooth curves, a stable meromorphic form can have a zero component, but not be zero globally. It can even have zero components and poles. On the complement of the nodes, zeros and poles, integration defines a translation structure. The underlying flat metric has the structure of a conical point angle $2\pi(1+\text{ord}_z(\omega))$ at a point $z$ of positive order or of an infinite half cylinder $\mathbb{H}^{\pm}/\mathbb{Z}$ at a simple pole of residue  $\pm 1$.  
At a node the metric is either a union of two conical points of possibly different angle (if the residue of the node is zero) or the disjoint union of two infinite half cylinders.

  \begin{definition}
  $\Omega_{0}^{\pm}(C,x_-,x_+)$ is the subset of forms in $\Omega^{\pm}(C,x_-,x_+)$ with zero residues at non-separating nodes.
  $\Omega_{0}^{\pm *}(C,x_-,x_+)$ is the subset of forms in $\Omega^{\pm}_{0}(C,x_-,x_+)$ with isolated zeros.  
  \end{definition}
 If $\omega\in \Omega^{\pm}_{0}(C,x_-,x_+)$ has some separating node, then its residue is, up to sign, zero or one  depending on whether $x_-,x_+$ lie on the same side of the node or not, i.e. depending on whether the node also separates the poles or not. 
 
Remark that the form $\omega$ can be integrated along any path avoiding the poles of the restrictions to components of $C^*$, hence any path avoiding $x_-,x_+$ and the nodes that separate the poles $x_-,x_+$. Remark also that, thanks to the Mayer-Vietoris sequence, any homology class in $H_1(C\setminus\{x_-,x_+\},\mathbb{Z})$ is represented by some path avoiding points with non-zero residue of $\omega$ (see the Appendix for details).

 The rank $g$ affine bundle \begin{equation}\label{eq:bundle of mero2simplepoles}\Omega^\pm\overline{\mathcal{M}}_{g,2}\rightarrow \overline{\mathcal{M}}_{g,2}\end{equation} is the bundle that has $\Omega^{\pm}(C,x_-,x_+)$  as fiber over the point $(C,x_-,x_+)$. It is a natural holomorphic extension of the affine bundle $\Omega^{\pm}\mathcal{M}_{g,2}\rightarrow \mathcal{M}_{g,2}$, directed by the Hodge vector bundle $\Omega\overline{\mathcal{M}}_{g,2}\rightarrow\overline{\mathcal{M}}_{g,2}$ of stable (holomorphic) forms. 
An element in $\Omega^\pm\overline{\mathcal{M}}_{g,2}$ will be denoted as a tuple $(C,x_-,x_+,\omega)$ or, since the information of the poles is already in $\omega$, by $(C,\omega)$. In \cite{5A18, 5A19} the reader can find more details on the structure of bundles of meromorphic one-forms over the Deligne-Mumford compactification and on Augmented Teichm\" uller space.  We are going to describe the relevant properties for our interest in the covers where the period map is naturally defined. 

The fiber bundle \eqref{eq:bundle of mero2simplepoles} can be pulled back via branched covers to topological bundles $$\Omega^\pm\overline{\mathcal{S}}_{g,2^*}\rightarrow\overline{\mathcal{S}}_{g,2^*}\text{ and  } \Omega^{\pm}\overline{\mathcal{S}}_{g,\text{\sout{2}}}\rightarrow\overline{\mathcal{S}}_{g,\text{\sout{2}}}\quad .$$ 
Elements in those covers will be denoted by $(C,[f],\omega)$ where $[f]$ denotes an equivalence class of topological marking under the covering group that defines the space.


Given a subset $\mathcal{K}$ of $\overline{\mathcal{M}}_{g,2}$ or any of its covers we denote by $\Omega^\pm\mathcal{K}$ the restriction of the $\Omega^{\pm}$-bundle  to the given subset. Its intersection with the set of forms having zero residues at all non-separating nodes will be denoted by  $\Omega^{\pm}_{0}\mathcal{K}$, and the intersection of the latter with the subset of forms isolated zeros by $\BOmega\mathcal{K}$.

On $\Omega_{0}^\pm\overline{\mathcal{S}}_{g,2^*}$ there is a well defined period map (see Appendix for details)

$${\mathcal Per}_{g,2^*}:\Omega_{0}^\pm\overline{\mathcal{S}}_{g,2^*}\rightarrow \text{Hom}(H_1(\Sigma_{g}\setminus\{p_-,p_+\}, \mathbb{Z});\mathbb{C})$$ that is equivariant with respect to the action of $\text{Mod}(\Sigma_{g,2})$ on source and target. It is defined by associating, to any class in the homology group, and a fixed marked form in $\Omega_{0}^\pm\overline{\mathcal{S}}_{g,2^*}$,  the integral of the form along the corresponding class in the stable curve (up to an appropriate choice of representative to allow for integration). The correspondence of classes is induced by the marking in homology. By construction the peripheral module $\Pi_2$ has rank one and its generator is sent to $\pm 1\in\mathbb{C}$.  

Similarly, we can define the period of an element in $\Omega_{0}^\pm\overline{\mathcal{S}}_{g,\text{\sout{2}}}$ as follows

\begin{definition}
 Given $(C,m,\omega)\in \Omega^{\pm}_{0}\overline{\mathcal{S}}_{g,\text{\sout{2}}}$ having zero residues at non-separating nodes there is a well defined homomorphism $$\Per(C,m,\omega): H_1(\Sigma_g,\mathbb{Z})\rightarrow\mathbb{C}/\mathbb{Z}\quad \text{ by } \gamma\mapsto \int_{m(\gamma)}\omega\quad \text{ mod }1$$
\end{definition}
\begin{definition}
 The  period map on $\Omega_{0}^{\pm}\overline{\mathcal{S}}_{g,\text{\sout{2}}}$ is the map  $$\percompact_{g}:\Omega^{\pm}_{0}\overline{\mathcal{S}}_{g,\text{\sout{2}}}\rightarrow H^1(\Sigma_g,\mathbb{C}/\mathbb{Z})$$ sending each point to its period homomorphism. 
\end{definition}
It extends the map $\per_g$ defined by Proposition \ref{p:period map on smooth curves} to a part of the boundary of $\Omega^{\pm}\overline{\mathcal{S}}_{g,\text{\sout{2}}}$ where it makes sense to talk about 'having the same period homomorphism' as a form in $\Omega^{\pm}\mrs_{g,\text{\sout{2}}}$. 

\begin{proposition}\label{p: period fiber cover}
The branched cover $\Omega^\pm\overline{\mathcal{S}}_{g,2^*}\rightarrow \Omega^\pm\overline{\mathcal{S}}_{g,\text{\sout{2}}}$ sends period fibers to period fibers. Its restriction to a fiber of $\percompact_{g,2^*}$ is a homeomorphism onto a fiber of $ \percompact_g$. 
\end{proposition}
In particular the branch points of  $\Omega^\pm\overline{\mathcal{S}}_{g,2^*}\rightarrow \Omega^\pm\overline{\mathcal{S}}_{g,\text{\sout{2}}}$ do not occur at points with zero residues at the non-separating nodes (compare with Figure \ref{fig:nodal branch point}).
\begin{proof}
The first part of the statement is a consequence of the isomorphism in \eqref{eq:peripheral ex seq} and the fact that the condition on the residues implies that the peripheral module $\Pi_2\subset H_1(\Sigma_g\setminus\{p_-,p_+\})$ has periods in $\mathbb{Z}$.
Remark also that, by \eqref{eq:special cover}, the monodromy of the cover on the unbranched part is $\text{Hom}(H_1(\Sigma_g),\Pi_2)$ and the branch points correspond to some fixed point of this action. However, the action of the non-trivial elements of the covering group does not leave any fiber of the period map invariant. Indeed, the action of $h\in\text{Hom}(H_1(\Sigma_g),\Pi_2)$ on a homomorphism $p:H_1(\Sigma_{g,}\setminus \{p_-,p_+\})\rightarrow \mathbb{C}$ is the homomorphism $p\circ (\text{Id}+h\circ \pi)$ where $$\pi:H_1(\Sigma_g\setminus\{p_-,p_+\})\rightarrow\frac{H_1(\Sigma_g\setminus\{p_-,p_+\})}{\Pi_2}\cong H_1(\Sigma_g).$$ Suppose $p$ is the period of some element in $\Omega^\pm\overline{\mathcal{S}}_{g,2^*}$. In particular $p_{|\Pi_2}:\Pi_2\rightarrow\mathbb{Z}$ is injective. Suppose $p$ is fixed by the action of $h$. Then $p+p\circ h\circ \pi=p$,   Since the image of $h$ lies in $\Pi_2$ we deduce $h\circ \pi=0$. But since $\pi$ is surjective, this implies $h=0$.  

Let $F$ be a non-empty fiber of ${\mathcal Per }_{g}$ and $\hat{F}$ its preimage under the branched map, a union of fibers. 
The restriction of the branched cover to  $\hat{F}\rightarrow F$
has no branch points and trivial monodromy. This concludes the proof.   
\end{proof}
\subsection{The stratification of augmented Teichm\" uller space and period fibers}
The stratification of the ambient space $\Omega^\pm\overline{\mathcal{S}}_{g,2^*}$ induces a partition of the fibers of the period map. The next theorem describes this partition

\begin{theorem}\label{t:local structure special cover}

The local fiber $L$ of the period map in $\Omega^\pm\overline{\mathcal{S}}_{g,2^*}$ (resp. $\Omega^{\pm}\overline{\mathcal{S}}_{g,\text{\sout{2}}}$) at a point with isolated zeros projects to the orbifold chart of $\Omega\overline{\mathcal{M}}_{g,2}$ as a complex manifold transverse to all boundary divisors through the point. Therefore, $L$ is  an abelian ramified cover of a normal crossing divisor in $(\mathbb{C}^{2g-2},0)$ having precisely one component of codimension one in each component of codimension one of the ambient space through the point. 
\end{theorem}

The statement is not true in general for forms with zero components.
\begin{proof}
The proof in $\Omega^\pm\overline{\mathcal{S}}_{g,2^*}$ (and more generally for $\Omega_{0}\overline{\mathcal{S}}_{g,n^*}$ with $n\geq 2$, i.e. for forms with at least two simple poles whose non-serparating nodes have zero residue)  follows exactly the same lines of the proof of the equivalent Theorem for abelian differentials in Section 4.18 of \cite{CDF2}. For completeness and future reference, we include a proof  in the appendix. 

As for the case of $\Omega^{\pm}\overline{\mathcal{S}}_{g,\text{\sout{2}}}$, the restriction of the homeomorphism of Proposition \ref{p: period fiber cover} to the neighbourhood of a point provides a homeomorphism that preserves boundary points. Hence the result. \end{proof}
\begin{corollary}\label{c:conn of closure equiv to conn}
Let  $F\subset\Omega^{\pm}\mrs_{g,\text{\sout{2}}}$ be a fiber of $\per_g$ and $\overline{F}$ its closure in $\Omega^{\pm*}_{0}\overline{\mathcal{S}}_{g,\text{\sout{2}}}$. Then $F$ is connected if and only if $\overline{F}$ is connected. \end{corollary}

\begin{proof}
The homeomorphism of Proposition \ref{p: period fiber cover} sends boundary points to boundary points, so it suffices to prove the equivalent statement for a fiber of the period map on the cover  $\Omega^\pm\mathcal{S}_{g,2^*}$ and its closure in the space  $\Omega_{0}^{\pm*}\overline{\mathcal{S}}_{g,2^*}$ of homologically marked meromorphic stable forms with zero residues at non-separating nodes and isolated zeros. This equivalent statement is proven in Corollary \ref{c:connected closure} of the Appendix. 

\end{proof}

\begin{definition}
 The closure of the fiber of $\per_g$ in $ \Omega^{\pm*}_{0}\overline{\mathcal{S}}_{g,\text{\sout{2}}}$ is what we consider as bordification of the fiber. It coincides with the fiber of $\percompact_g$ on $ \Omega^{\pm*}_{0}\overline{\mathcal{S}}_{g,\text{\sout{2}}}$. 
 \end{definition}
\subsection{Smoothings, simple boundary points,  and the proof of Theorem \ref{t:realization}}
 \begin{definition}
 A path in $\Omega_0^{\pm }\overline{\mathcal{S}}_{g,\stkout{2}}$ is said to be isoperiodic if the value of $\percompact_g$ is constant along it (i.e. it lies in a fiber of $\percompact_g$). To simplify the notations, two marked forms $\omega_1$ and $\omega_2$ in $\Omega_{0}^{\pm*}\overline{\mathcal{S}}_{g,\stkout{2}}$ are equivalent and denoted $\omega_1\sim\omega_2$ if they can be joined by an isoperiodic path in the total space. 
\end{definition}

Recall that every node of a stable curve $C$ determines a boundary component of $\overline{\mathcal{S}}_{g,\stkout{2}}$ passing through $C$.
\begin{definition}
A smoothing of a node of a stable form is an equivalent form that does not belong to the boundary component corresponding to the node but that belongs to the other boundary components passing through the point.
\end{definition}
The metric description of the neighbourhood of the node when smoothing it is a flat cylinder of finite area. If the residue at the node is non-zero, when we get closer to the form $\omega$, the  circumference of the annulus stays constant but its volume tends to grow, so the cylinder becomes long. 

\begin{definition}
 A smoothing of a boundary point $\omega$ with isolated zeros is an equivalent form on a non-singular curve. They do always exist thanks to Theorem \ref{t:local structure special cover}.
 \end{definition}

 Remark that the dual graph of a curve of compact type supporting a meromorphic form  with two simple poles, each of whose nodes separate the two poles is a segment. We denote such a form as an \textit{ordered} sequence $\omega_1\vee\cdots\vee\omega_k$ where $\omega_i$ is a meromorphic form with two poles on a smooth curve $C_i$, $\omega_1$ (resp. $\omega_k$) has a pole of residue $-1$ (resp. $+1$) and the pole of residue $+1$ of $\omega_j$ is glued to the pole of residue $-1$ of $\omega_{j+1}$ to form a node. Each $\omega_i$ is called a part of the form and has isolated zeros.

The maps \eqref{eq:attaching at torelli level} extend at the level of forms:

\begin{equation}
\label{eq:attaching forms at torelli level}\Omega^{\pm}\mrs_{g_1,\stkout{n_1\cup*}}\times\Omega^\pm\mrs_{g_2,\stkout{n_2\cup\star}}\rightarrow\Omega_0^{\pm *}\overline{\mathcal{S}}_{g_1+g_2,\stkout{n_1+n_2}}\end{equation}  by the relation \begin{equation}\label{eq:attach torelli forms cpt type}(C_1,Q_1\cup*,m_1,\omega_1)\times(C_2,Q_2\cup\star,m_2,\omega_2)\mapsto (C_1\vee_{*=\star} C_2,Q_1\cup Q_2, m_1\oplus m_2,\omega_1\vee\omega_2).\end{equation} by imposing that the two marked points where we glue are of opposite residue. In particular  if one of the poles is chosen chosen as marked point, the other is completely determined.

\begin{remark}
The map \eqref{eq:attaching forms at torelli level} is continuous and sends products of period fibers to period fibers. In particular it sends products of isoperiodic paths to an isoperiodic path. We use the following shorthand notation \begin{equation}\label{eq:attaching preserve equivalence}
\omega_{1}\sim\omega_1',\quad \omega_{2}\sim\omega_2'\Rightarrow\omega_{1}\vee\omega_{2}\sim\omega'_{1}\vee\omega'_{2}
\end{equation}
\end{remark}
A generalization of the following remark motivates the inductive proof of Theorems \ref{t:realization} and \ref{t:Cgn}: 

\begin{remark}\label{rem:realization with poles on elliptic curves}
Given a holomorphic form $\omega$ on an elliptic curve $E$ the form $\omega\vee\frac{dz}{2i\pi z}$ on the stable curve $E\vee\Pone$ defines a boundary point in $\Omega_0^{\pm^*}\overline{\mathcal{S}}_{1,\stkout{2}}$, and the same periods in $\mathbb{C}/\mathbb{Z}$ as $\omega$. The smoothing of the latter defines an isoperiodic form with two poles on a smooth elliptic curve. 
\end{remark}

\begin{definition} 
 A boundary point of $\Omega_0^{\pm*}\overline{\mathcal{S}}_{g,\stkout{2}}$ is said to be simple if the underlying nodal curve is of compact type, each node separates the poles, and the degree of each part is at least three. 
\end{definition}

Simple boundary points do not exist for fibers corresponding to periods of degree two.  Let us analyze the existence of simple boundary points for other fibers.
Recall the following definition
\begin{definition}
Let $g\geq 2$ and $p:H_1(\Sigma_g)\rightarrow A$ be non-trivial homomorphism of abelian groups. A symplectic decomposition $H_1(\Sigma_g)=V_1\oplus\cdots\oplus V_k$ is said to be $p$-admissible if none of the restrictions $p_{|V_i}$ is trivial. 
\end{definition}
A simple boundary point $(C_1\vee\cdots\vee C_k,m,\omega_1\vee\ldots\vee\omega_k)$ induces a $[p]$-admissible \textit{ordered} decomposition of the homology group $H_1(\Sigma_g)$ $$V_1\oplus \cdots\oplus V_k\text{ where } V_i=m^{-1}(H_1(C_i))$$  where $[p]$ is the class modulo $\frac{1}{2}\mathbb{Z}$ of the period homomorphism  $$p=\percompact_g(C_1\vee\cdots\vee C_k,m,\omega_1\vee\ldots\vee\omega_k):H_1(\Sigma_g)\rightarrow \mathbb{C}/\mathbb{Z}.$$ The order of the factors is determined by the order of the forms, baring in mind that the first one (resp. last one) has a pole of $-1$-residue (resp. $+1$-residue)
\begin{lemma}\label{l:existence of adm decomp}
For $g\geq 2$ and $p:H_1(\Sigma_g,\mathbb{Z})\rightarrow \mathbb{C}/\mathbb{Z}$ non-trivial, there exists a $p$-admissible decomposition. If $\deg p\geq 3$ there exists a $p$-admissible decomposition whose factors have also degree at least three.

Any such decomposition is induced by some stable meromorphic form with two poles on a curve of compact type whose nodes separate the two poles. 
\end{lemma}
\begin{proof}
The algebraic part is proven in  Corollary \ref{c:padmissible decompositions exist}. 
By inductively applying the algebraic result we can find a sub-decomposition $H_1(\Sigma_g,\mathbb{Z})=V_1\oplus\ldots \oplus V_k$ of any given decomposition into symplectic submodules of rank two such that $p_{|V_i}$ is non-trivial. Lemma \ref{l: computation of the isoperiodic foliation in genus one} implies that there is either a form with two simple poles or a holomorphic form $\omega_i$ realizing $p_{|V_i}$. By Remark \ref{rem:realization with poles on elliptic curves} we can suppose that $\omega_i$ has two simple poles.  
By considering forms on smooth genus one  curves $E_i$  with two poles  $\omega_i$ and periods $p_i$ we can construct the boundary form of compact type $\omega_1\vee\ldots\vee\omega_k$ of period $p$ by connecting the negative pole of one with the positive of the next to a stable meromorphic form.  
Smoothing the relevant nodes, we can realize any decomposition obtained by joining factors of the one induced by $V_i's$.

\end{proof}

\begin{proof}[Proof of Theorem \ref{t:realization}]
For $g=1$ we use Remark \ref{rem:realization with poles on elliptic curves}. For $g\geq 2$ we realize $p$ as the period of a form on a marked stable curve of compact type by using Lemma \ref{l:existence of adm decomp}. By smoothening its nodes, we obtain the desired form on a smooth curve.
\end{proof}




\section{Degenerating to a simple boundary point}
\label{s:degenerating to a simple boundary point}
A Schiffer variation (see \cite{CDF2} and references therein for details) allows to define isoperiodic paths in $\Omega^\pm\overline{\mathcal{M}}_{g,n}$ (and its cover $\Omega_0^{\pm*}\mrs_{g,\stkout{2}}$) by deforming any marked stable meromorphic form having an isolated zero continuously in its fiber of the period map. It is defined by a continuous surgery for any given family of twin paths, i.e. paths $\gamma_1,\gamma_2, \ldots,\gamma_k$ with parameter space $[0,1)$  starting at an isolated zero (nodes with zero residue are allowed) of a meromorphic form $(C,\omega)$ such that the $k$ paths in $\mathbb{C}$ defined by integration $$t\mapsto\int_{\gamma_{i|[0,t]}}\omega\quad\text{ coincide for }\quad i=1,\ldots,k.$$ It actually deforms the flat structure underlying the form. For more details and examples of deformations that include nodal curves, we refer to \cite{CDF2}. If the form is marked, the surgery allows to follow the marking along the surgery. In this section we use Schiffer variations to connect any meromorphic form to some form on a nodal curve. 



\begin{proposition}\label{p:degeneration2}
 Any meromorphic form on a smooth genus $g\geq 2$ Riemann surface with two simple poles of residues $\pm 1$ can be deformed via a sequence of Schiffer variations to a stable meromorphic form on a singular stable curve with one node and isolated zeros. If the periods are real we can further suppose that the node is separating and separates the poles. Otherwise we can suppose that the node has zero residue.
 \end{proposition}

\subsection{Proof of Proposition \ref{p:degeneration2} in the case of real periods}

 The integral of the imaginary part of a meromorphic form $\omega$ with real periods and two poles provides a harmonic surjective first integral $h_{\omega}:X\rightarrow \mathbb{R}$ for the horizontal foliation. We claim that up to a Schiffer variation we can suppose that $h_{\omega}$ has a regular connected fiber (that is necessarily a closed leaf) that separates $X$ in two surfaces of positive genera. The cylinder it generates has a height. By moving all zeroes in each side of the leaf, in the imaginary direction, towards the pole in the corresponding side, at the same speed, we have that the height of the cylinder of closed leaves generated by the chosen leaf is growing. In the limit we get a stable form on a singular nodal curve with a separating node none of whose components is the zero form. 
 \begin{figure}[h]
     \centering
     \includegraphics[height=7cm]{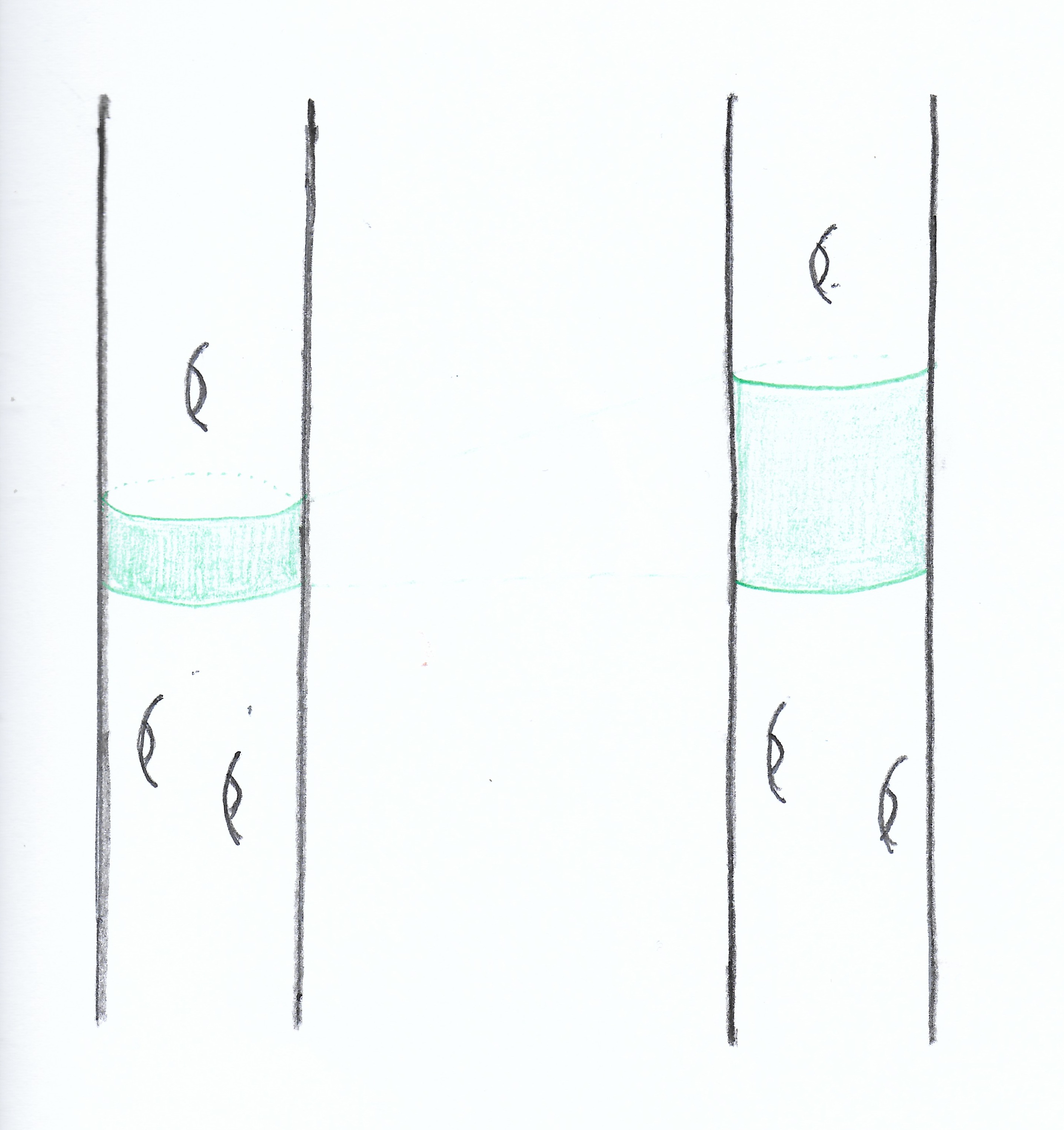}
     \caption{Isoperiodic stretching of a cylinder by appropriate Schiffer variations}
     \label{fig:graphofcylinders}
 \end{figure}
 
 To get the desired closed leaf, first apply some Schiffer variations in the imaginary direction to guarantee that all zeroes of $\omega$ are simple and have different values $-\infty<x_1<x_2<\ldots<x_{2g}<\infty$  of the first integral $h_{\omega}$. The fiber $h_\omega^{-1}(x)$ over a point $x\neq x_i$ is a disjoint union of circles. The number of circles in each fiber is constant in each interval of $\sqcup_{i=0}^{2g}I_i=\mathbb{R}\setminus\{x_1,\ldots,x_{2g}\}$ and generates cylinders of closed leaves. In each unbounded interval it corresponds to one of the cylinders defined by a pole. Since all zeroes are simple, the number of cylinders changes at each $x_i$ by a unit: on one side there is one cylinder approaching, on the other, two. Hence, over $I_1$ and $I_{2g-1}$ there are two connected components of the fiber. 
 
 The graph of horizontal cylinders of $\omega$ is the trivalent connected graph having $2g$ vertices (one for each zero) and we put an edge between two zeroes if there is a cylinder of closed leaves that has both zeroes in the boundary. (see Figure \ref{fig:graphofcylinders2} for an example).
 \begin{figure}[h]
     \centering
     \includegraphics[height=7cm]{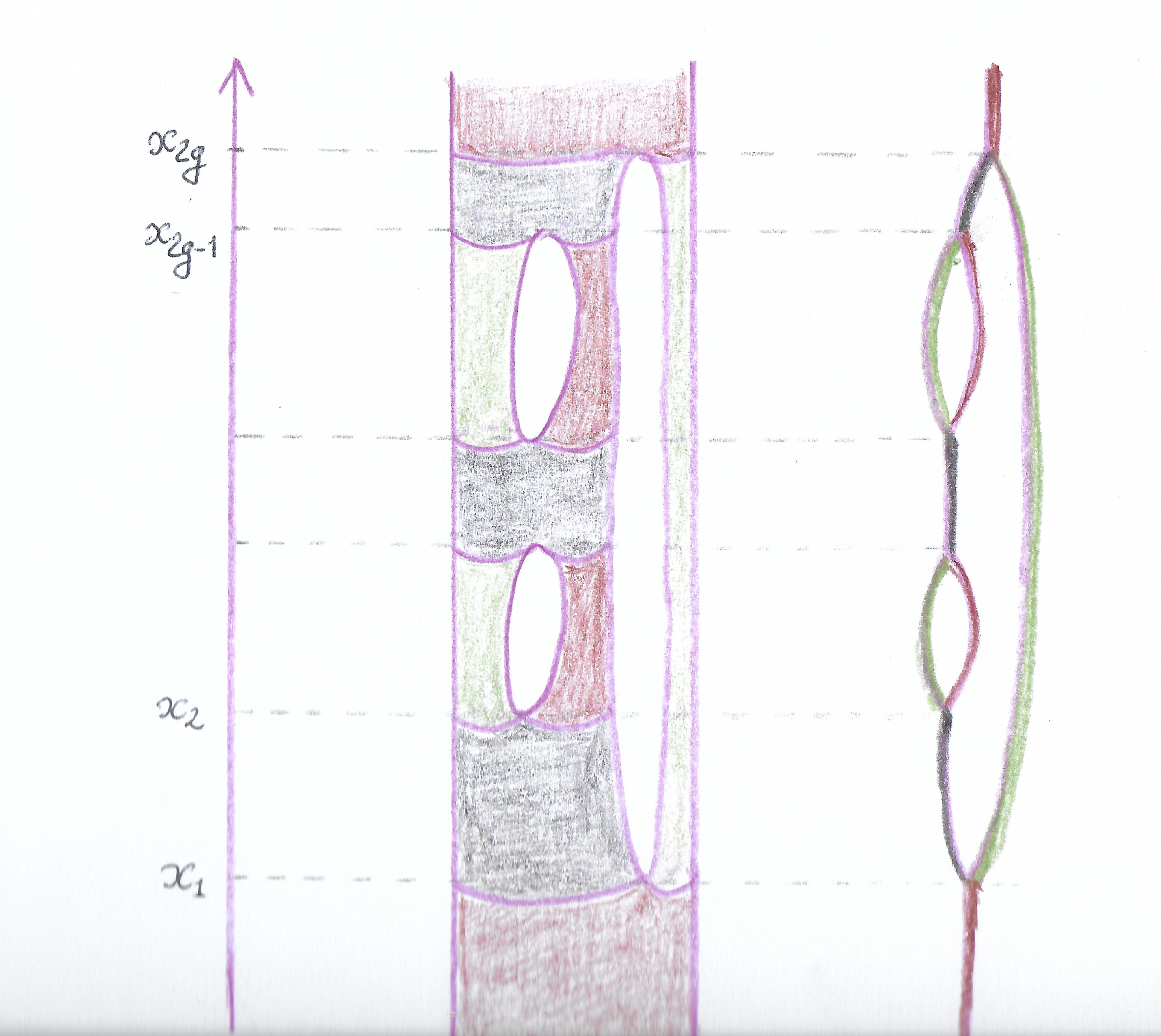}
     \caption{An example of the graph of cylinders}
     \label{fig:graphofcylinders2}
 \end{figure}

 If there is an intermediate segment $I_2,\ldots,I_{2g-2}$ for which the fiber $h_{\omega}^{-1}(I_j)$ is one cylinder we will be done. 
We will analyze how the graph of horizontal cylinders changes along a movement of zeroes, via Schiffer variations and verify that up to this further surgery we fall in the desired situation.

Let us analyze the change of graph we can obtain by applying a particular Schiffer variation at a zero \(z_0\) that has two components going upwards. We denote by \(v_0\) the vertex of the cylinder graph corresponding to \(z_0\). Among the two edges \(e_0, e_0 '\) that go from \(v_0\) upward, one has smallest length, say \(e_0\). Its other extremity is denoted \(v_1\). Let \(e_1\) be an edge going upward at \(v_1\). There exists a pair of twins at \(z_0\) whose projections to the cylinder graph are 
\begin{enumerate}
    \item the concatenation of \(e_0\) with a small interval in \(e_1\), 
    \item a portion of \(e_0\). 
\end{enumerate}
Indeed, first choose a geodesic starting at \(z_0\) whose projection to the cylinder graph satisfies (1), and then take its twin. The Schiffer variation along this pair of twin produces a new meromorphic differential whose cylinder graph is depicted in Figure \ref{fig:graphaftershiffer}.  

\begin{figure}[h]
     \centering
     \includegraphics[height=6cm]{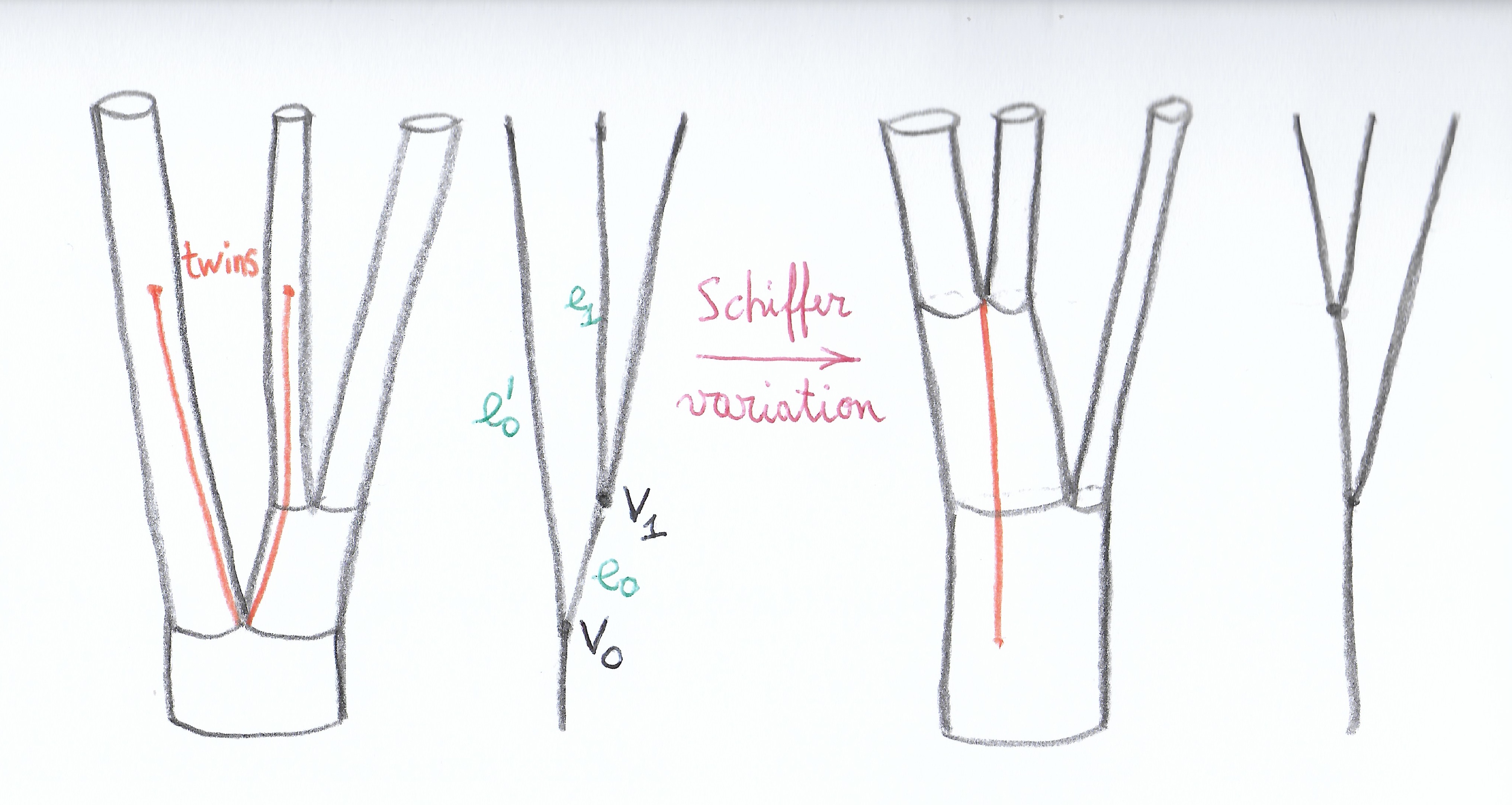}
     \caption{Graph changes when a zero moves upward.} 
     \label{fig:graphaftershiffer}
 \end{figure}


 Choose two paths in the graph, each leaving the top vertex of the cylinder graph on different edges, and choose any way of going down the graph with each of them. At the first moment they meet (they always do, since the graph ends at a unique vertex) we find a point $y_0$ that has two edges upward. Let \(v\) be the number of vertices that we find in the union of the paths.  If $v=2$, we are done, since there is a cylinder just under $y_0$. Use the choice of paths to move $y_0$ up the graph. As explained in Figure \ref{fig:graphaftershiffer}, we can move the point $y_0$ upwards along the given choice of paths. At each step we transform the cylinder graph, and reduce the number $v$ by one. Proceeding inductively we reach the situation where this number is two.

\subsection{Degeneration in the case of non-real periods}



The proof works in a more general case. Suppose $(C,m,\omega)$ is a meromorphic form with simple poles on a smooth connected curve, whose residues at the poles lie in a real line. Suppose further that there is some period not lying in the same real line. In particular this is the case for the case of two simple poles of residues $\pm 1$ and non-real periods. 

We claim that by Schiffer variations we can either find a pair of distinct non-closed twin paths with common end-point (at a zero) or join all zeroes to a point. By changing the point locally in the period fiber using period coordinates in the stratum (see Remark 3.8 in \cite{boissy}), we can suppose that the coordinates corresponding to saddle connections between distinct zeroes of $\omega$ have different values in $\mathbb{C}$, and, moreover,  none coincides with the period of a cycle, i.e. lies in the (countable) image of the period map $\per(C,m,\omega)$ in $\mathbb{C}$. Consider one of the shortest saddle connection, and the set of all its twins at its starting point. By the choices made so far, the only possibilities for endpoints of any of those paths is a regular point, or the endpoint of the original saddle connection. If one of the twin paths has the same endpoint, we have found a pair of twins whose endpoints coincide. Otherwise, the Schiffer variation along the set of all twins at the point produces a meromorphic differential with one zero less (of higher order).  Continuing the process with the obtained form inductively we end up in one of the desired situations. 
 If we have found the couple of twins between distinct zeroes, the Schiffer variation will degenerate the surface to a node with zero residue and no zero components and we are done.

Next we claim that if the form has a single zero, we can also find, up to some Schiffer variations, a pair of twins having the same endpoints (but this time the starting point and endpoints might coincide also). Indeed, consider the horizontal foliation induced by $\omega$ on $C$. Around each pole it has an open cylinder of infinite volume formed by closed regular geodesics. The complement $\mathcal{U}\subset C$ of the cylinders in $C$ has finite volume. It cannot be empty since the form $\omega$ has some non-real period. The boundary is formed by a collection of saddle connections at the unique saddle of $\omega$.  According to \cite[Proposition 5.5]{Tahar1} in each connected component of $\mathcal{U}\setminus\{\text{saddle connections}\}$ the real foliation is either a finite volume cylinder formed by closed regular geodesics, or minimal. We claim that, up to changing the direction of the foliation we can find a cylinder by closed regular geodesics. Indeed, if there are none for the real foliation, we can apply Lemma 5.12 in \cite{Tahar1}  to one of the minimal components and find a finite volume cylider in $\mathcal{U}$ formed by closed regular geodesics (necessarily in a different direction). 

The boundary curves $b^+$ and $b^-$ of a finite volume cylinder by closed geodesics have the same length. Each defines a finite sequence of saddle connections based at the zero of $\omega$. In this situation up to some Schiffer variations, we can always find a pair of twin curves having the same endpoints. It suffices to carry exactly the same proof by cases of section 5.3 in \cite{CDF2}. As in that case, the Schiffer variation along this pair of twins degenerates the surface to an isoperiodic stable form  with no zero components and a node of zero residue.

\subsection{From boundary points to simple boundary points}

\begin{proposition}\label{p:degenerating to simple boundary point}
Let $g\geq 2$. Then any form in $\Omega_0^{\pm*}\mnc_{g,\stkout{2}}$ with period homomorphism of degree at least three is equivalent to a simple boundary point. 
\end{proposition}






\begin{proof}
By induction on the genus. 
Let $p:H_1(\Sigma_2,\mathbb{Z})\rightarrow \mathbb{C}/\mathbb{Z}$ with $\deg (p)\geq 3$. Since $\per_2^{-1}(p)$  is connected by Corollary \ref{c: connectedness genus two} it is enough to prove that there exists a simple boundary point in the bordification. 

By Lemma \ref{l:existence of adm decomp} there exists a decomposition $H_1(\Sigma_2,\mathbb{Z})=W_1\oplus W_2$ into rank two symplectic submodules such that each $p(W_i)$ has at least three elements. Up to identifying $W_i$ with $H_1(\Sigma_1,\mathbb{Z})$,  $p_{|W_i}$ can be thought as an element in $H^1(\Sigma_1,\mathbb{C}/\mathbb{Z})$ of degree at least three. By Theorem \ref{t:realization} we can consider a form $\omega_i$ on a smooth curve with two poles and periods $p_{|W_i}$.
Hence  $\omega_1\vee\omega_2$ is a simple boundary point of the fiber $\per_2^{-1}(p)$.

Fix some $g\geq 3$ and suppose that the statement of the proposition is true up to genus $g-1$. Consider $p:H_1(\Sigma_g,\mathbb{Z})\rightarrow \mathbb{C}/\mathbb{Z}$ with $\deg (p)\geq 3$. 

Along this proof, and by abuse of notation, we will denote the stable form obtained by gluing two meromorphic forms $\omega_1$,$\omega_2$ on smooth curves $C_1,C_2$ glued at points that are poles or not by $\omega_1\vee\omega_2$ and omit the information of the points where they are glued. A smoothing of such a form will be denoted by $\omega_1\widetilde{\vee}\omega_2$. 
\vspace{0.5cm}

\underline{Case 1:} Suppose $\im (p)\subset \mathbb{R}$ \\
Up to applying Proposition \ref{p:degeneration2} we can suppose that we start at a form $\omega=\omega_1\vee\omega_2$ with precisely one node that separates the surface and has one pole on each side. The decomposition of the form provides a decomposition $p=p_1\oplus p_2$. Suppose that one of $p_i$ has degree two. Since $\deg p\geq 3$ , we cannot have both of degree two. If the component which is not of degree two, say $\omega_1$,  is of genus at least two, we can use the hypothesis to join  $\omega_1$ to a nodal stable form $\omega_3\vee\omega_4$ with non-zero residue at the node, and each part of which has periods of degree at least three. The form $\omega_2\vee\omega_3$ has periods of degree at least three, and so does $\omega_4$. Then $\omega\sim(\omega_2\widetilde{\vee}\omega_3)\vee\omega_4$. 

If the form $\omega_1$ (of degree at least three) has genus one, then $\omega_2$ has genus $g-1$ and degree two. We can still find an equivalent $\omega_3\vee\omega_4$ each part of which will have degree two. However, the sum $\omega_1\vee\omega_3$ will have degree at least three and genus at least two. The previous case can be applied to the form  $(\omega_1\widetilde{\vee}\omega_3)\vee\omega_4$, that is equivalent to $\omega$.\\
\vspace{0.4pt}

\underline{Case 2:} Suppose $p$ has a non-real number in the image.

\vspace{0.5pt}

\underline{Subcase 2.1} From a non-simple boundary point to either a simple boundary point or a form on a curve with a non-separating node with zero residue. \\

Up to applying Proposition \ref{p:degeneration2} we can suppose that we start at a form $\omega$  with precisely one node with zero residue. If it is non-separating we are done. Suppose the node is separating and write $\omega=\omega_1\vee\omega_2$. 
Let  $g_i$ be the genus of $\omega_i$.
One of the parts, say $\omega_2$, has no poles and the other has two poles. Therefore  $\omega_2$ has infinite degree. 

If $g_1\geq 2$ and $\omega_1$  has degree at least three then we can connect $\omega_1$ to a simple boundary point $\omega_3\vee\omega_4$ in genus $g_1$. The form $\omega$ is therefore connected to $\omega_3\vee(\omega_4\widetilde{\vee}\omega_2)$ which has a separating node with non-zero residue and degree at least three on each part. It corresponds to a simple boundary point. 

If $g_1=1$ then $g_2\geq 2$ and we can connect $\omega_2$ to a stable holomorphic form on a nodal curve by an isoperiodic path  (see section degeneration in \cite{CDF1}). If the node is non-separating we are in the second possibility of subcase 2.1. Otherwise, we can write this nodal curve as $\omega_3\vee\omega_4$ both holomorphic with non-zero components. The form $\omega$ is connected to $(\omega_1\widetilde{\vee}\omega_3)\vee \omega_4$ which falls in the previous case. 

 Next we suppose that  $\omega_1$ has period map of degree two. Up to using Lemma \ref{p:degeneration2} repeatedly to $\omega_1$ we can further suppose $g_1=1$ and therefore $g_2\geq 2$. Since the periods $p$ are not all real,  $\omega_2$ has non-real periods and has infinite degree. If it has no poles we fall in a previous case. If it has poles, we can use the hypothesis to connect $\omega_2$ to $\omega_3\vee\omega_4$ with both parts of degree at least three and two poles. 
 The form $(\omega_1\widetilde{\vee}\omega_3)\vee\omega_4$ is a simple boundary point that can be connected to $\omega$ by an isoperiodic path.

 \vspace{0.5cm}
\underline{Sub-case 2.2}
From a form with zero residue on a non-separating node to a simple boundary point. \\

Given a form without zero components $\omega$ with some non-separating nodes of zero residue, we denote by $\widetilde{\omega}$ some form obtained by smoothing the non-separating nodes of  $\omega$.

Let $(C,m,\omega)$ be  a form on a curve with a single non-separating node of zero residue. We claim first that $\omega$ is equivalent to a stable form  $\omega_1\vee\omega_2$ where $\omega_1$ is a meromorphic stable form with two poles on a singular curve of genus one (containing the non-separating node of zero residue) and $\omega_2$ is a meromorphic form on a smooth curve with two poles. This can be achieved if we find a smooth closed geodesic for $\omega$ that separates the surface $C$ in two parts with the said properties. We claim that up to deforming $\omega$ by an isoperiodic path in its stratum we can find such a geodesic. Indeed, consider a simple closed curve $\gamma$ on $C$ with base-point at the node that defines a non-trivial homology class in $H_1(C,\mathbb{Z})$. Let $\gamma_0$ denote the path with distinct endpoints obtained in the normalization $(C_0,m_0,\omega_0)$ of $(C,m,\omega)$. Any continuous deformation $\gamma_t$ of $\gamma_0$ in the space of paths in $C_0$ with distinct endpoints such that $t\mapsto\int_{\gamma_t}\omega_0$ is constant allows to define an isoperiodic path starting at $(C,m,\omega)$. It suffices, for each $t$, to glue the endpoints of the path $\gamma_t$ to obtain a curve with a single non-separating node of zero residue $(C_t,\omega_t)$ and follow the same rule that marked $C_0$ from $C$ to mark the homology $C_t$ from $C_0$. By construction, geodesics of $\omega_0$ are sent to geodesics of $\omega_t$ for each $t$. Therefore, if me manage to find such a deformation $\gamma_t$ for which both endpoints lie in the cylinder around one of the poles of $\omega_0$ we will be able to find the closed geodesic as one of the geodesics of the cylinder leaving both endpoints at the same side of the pole. Consider the oriented directional foliation in $C_0$ defined by $\omega_0$ on a generic direction in $\mathbb{C}$ that is not tangent to the real line and such that the endpoints $q_1,q_2$ of $\gamma_0$ do not belong to saddle connections (see \cite{Strebel}{, section 11.4}). All geodesics in that direction that are not saddle connections converge to the same pole under the geodesic flow $G_t$ of the chosen direction. Let $t_1$ be a time such that $G_{t_1}(q_i)$ belongs to the interior of the infinite cylinder around that pole for $i=1,2$. Then for each $t\in[0,t_1]$, the path $\gamma_t$ obtained by joining $G_t(q_1)$ with $q_1$ along the geodesic, $\gamma_0$ and then $q_2$ with $G_t(q_2)$ along the geodesic, satisfies $\int_{\gamma_t}\omega_0=\int_{\gamma_0}\omega_0$ for all $t\in[0,t_1]$ and $\gamma_t$ has distinct endpoints. 

If $\omega_2$ has a period map of degree at least three we can write $\omega_2\sim\omega_3\vee\omega_4$ with each part of degree at least three and with non-zero residues. The form $\omega$ is equivalent to the form $(\widetilde{\omega_1}\widetilde{\vee}\omega_3)\vee\omega_4$ which determines a simple boundary point. 

If $\omega_2$ has degree two, then $\omega_1$ cannot have degree two (otherwise $\omega$ would). Write $\omega_2\sim\omega_3\vee\omega_4$ where now $\omega_3$ and $\omega_4$ have degree two. The form $\omega_1\vee\omega_3$ has periods of degree at least three over a curve of genus at least two. By inductive hypothesis it is equivalent to a form $\omega_5\vee\omega_6$  where each part has degree at least three. The initial form is equivalent to $\omega_5\vee(\omega_6\widetilde{\vee}\omega_4)$ has parts of degree at least three and non-zero residue. Hence determines a simple boundary point.

\end{proof}

\section{Sufficient conditions for the connectedness of the set of simple boundary points of a period fiber}

 \label{s:sufficient conditions}
\subsection{Equivalence of simple boundary points inducing the same factors} 


\begin{proposition}\label{p:equivalence with same factors}
Let $g\geq 3$ and $p\in H^{1}(\Sigma_{g},\mathbb{C}/\mathbb{Z})$ of degree at least three. Suppose Theorem \ref{t:Cgn} is true up to genus $g-1$.   Then any pair of $p$-simple boundary points  inducing the same factors $V_{1},\ldots, V_{k}$ in the decomposition of $H_{1}(\Sigma_{g})$ (maybe in different order) lie in the same connected component of the set of $p$-simple boundary points of the period fiber. 
\end{proposition} 

\begin{proof} 
The homomorphism  $p_{i}:=p_{|V_{i}}$ has degree at least three for each $i$.  Identify $V_{i}$ with the homology group $H_{1}(\Sigma_{g_{i}})$ for the appropriate $g_{i}<g$. By Theorem \ref{t:Cgn} applied to $p_i$,  the fiber $\per_{g_{i}}^{-1}(p_{i})$ is connected. 

Suppose first that the decompositions associated to the two simple boundary points have the  factors in the same order.  Then, both lie in the image of the same (restricted) attaching map $$\per_{g_{1}}^{-1}(p_{1})\times\cdots\times\per_{g_{k}}^{-1}(p_{k})\rightarrow \Omega^{\pm}\overline{\mathcal{S}}_{g,\stkout{2}}$$
that attaches following the same rule of attaching of the given forms: the $(+1)$-pole of a part is glued to the $(-1)$-pole of the next. Any form in the image has some node and period homomorphism $p_1\oplus p_2\oplus\cdots\oplus p_k=p$. Since the source is connected both points belong to the same connected component of $\percompact^{-1}(p)$ and are therefore equivalent.

Next we claim that  changing the order of the parts of a given form $\omega_{1}\vee\cdots\vee\omega_{k}$ of genus at least three with parts of degree at least three defines forms that lie in the same connected component of the boundary of $\per^{-1}(p)$. Up to applying Proposition \ref{p:degenerating to simple boundary point} to each of the parts and using \eqref{eq:attaching preserve equivalence} we can suppose that all the parts of the form have genus one. It suffices to check that when $g\geq 3$ we can change the order of two consecutive factors of genus one without leaving the boundary to deduce the result. 

Now, for any pair of forms $\omega_{1}, \omega_{2}$ of genus one with two poles of residues $\pm 1$, we know by Corollary \ref{c: connectedness genus two} that there is an isoperiodic path $\{\eta_{t}\}_{t\in[1,2]}$ joining $\omega_{1}\vee\omega_{2}$ with $\omega_{2}\vee\omega_{1}$. This path will, in general, leave the boundary. 
However for any other pair of forms $\omega, \eta$ (not both zero) with two poles of residues $\pm 1$ the path $\eta\vee\eta_{t}\vee \omega$ joins $\eta\vee\omega_{1}\vee\omega_{2}\vee\omega$ with $\eta\vee\omega_{2}\vee\omega_{1}\vee\omega$ along $p$-simple boundary points.

\end{proof}
\subsection{Proof of Proposition \ref{p:algebrization}}
It remains to find equivalences between simple boundary points whose associated decompositions have distinct factors.


The realization of decompositions given by Lemma \ref{l:existence of adm decomp} allows to algebraize the problem completely. Recall the definition of admissible decompositions given in Definition \ref{def:p-admissible}. 


By smoothing all except for one node, we reduce the problem to finding equivalences between simple boundary points with one node. For each such form we can associate a vertex of the graph of $[p]$-admissible submodules, namely the vertex corresponding to the pair $\{V,V^{\perp}\}$ where $V\oplus V^{\perp}=H_1(\Sigma_g)$ is the decomposition induced by the simple boundary point. By Proposition \ref{p:equivalence with same factors}, two simple boundary points having the same associated vertex are equivalent. It suffices to find equivalences between distinct vertices. 

If two simple boundary points are associated to distinct vertices $\{V_1,V_2\}$ and $\{V_1',V_2'\}$  that are joined by an edge in the graph of $[p]$-admissible decompositions, it means that there exists a decomposition $W_1\oplus\cdots\oplus W_k$ with at least 
three factors one of which is in the first pair and another in the other pair. Using Lemma \ref{l:existence of adm decomp}, let $\omega_1\vee\ldots\vee\omega_k$ be a simple boundary point of period homomorphism $p$ whose associated decomposition is $W_1\oplus\cdots\oplus W_k$. By smoothing the nodes appropriately we can find an equivalence between this form and a form having the factors $\{W_j,W_j^{\perp}\}$ for any $j=1,\ldots,k$ and period homomorphism $p$. Hence the  two simple boundary points are equivalent.

The connectedness of the graph of $[p]$-admissible decompositions implies that all simple boudary points are equivalent.

\section{Connectedness of graphs of $p$-admissible decompositions}
\label{s:algebraic results}
In this section we prove Theorem \ref{t:graph of admissible}. 

 We consider \(V\) a symplectic unimodular module over \(\Z\), of rank at least four, and \(p : V\rightarrow A \) a non zero morphism to an abelian group. (Ultimately, this morphism will be the reduction modulo \(\Z\) (resp. modulo \(\frac{1}{2} \Z\)) of the period of a meromorphic differential with two poles around which the periods are \(1\) and \(-1\), so \(A= \C/\Z \text{ or }\frac{\C}{\frac{1}{2}\Z}\)) 

Recall the definition of the graph of $p$-admissible decompositions and elements in Definitions \ref{def:p-admissible} and \ref{def:p-admissible element}. To lighten the notation we introduce an equivalence relation among $p$-admissible submodules. 

\begin{definition}
Two $p$-admissible submodules $V_1$ and $V_2$ of $V$ are said to be equivalent and denoted $V_1\sim V_2$ if the vertices corresponding to $V_1\oplus V_1^{\perp}$ and $V_2\oplus V_2^{\perp}$ lie in the same connected component of the graph of $p$-admissible submodules.  
\end{definition} 

We start with a corollary of Lemma \ref{l: admissible element}: 
\begin{corollary}\label{c: reduction to rank two}
Any \(p\)-admissible symplectic decomposition is equivalent to one with a rank two factor.
\end{corollary}
The rest of the section will be devoted to find equivalences between rank two $p$-admissible decompositions. 

\subsection{Forcing intersection}

\begin{lemma}\label{l: creating intersection}
Assume \(\text{rank} (V) \geq 6\). Then, given any pair \(W, W' \) of  \(p\)-admissible submodules of rank two, there exists another pair \( W_1  , W_1' \) which is such that 
\begin{itemize}
\item \(W_1\) and \(W_1'\) are \(p\)-admissible submodules of rank two,
\item \(W_1 \cap W_1 '\neq \{0\}\), 
\item \(W_1\sim W\) and  \(W_1'\sim W '\).
\end{itemize}
\end{lemma}

\begin{proof} 
The spaces  \( W^\perp\) and \((W')^\perp\) intersect on a linear subspace of dimension at least two.  In particular, the set  \( \left( W^\perp\cap (W')^\perp\right)  \setminus \left(NA(p_{|W^\perp}) \cup NA(p_{|(W')^\perp})\right) \) is non empty. By Lemma \ref{l: admissible element}, an element in this set belongs to the spaces \(W_1\) and \(W_1'\) of \(p\)-admissible symplectic decompositions \( W^\perp = W_1 \oplus  W_2 \)  and \( (W')^\perp = W_1' \oplus  W_2'\). The three items are satisfied for the pair \(W_1, W_1'\) and we are done.
\end{proof}

\subsection{Turning around a line}

\begin{lemma} \label{l: proof of equivalence I}
Assume that \(\text{dim} (V) \geq 6\). Let \(W_1, W_1'\) be a pair of \(p\)-admissible submodules  of rank two such that  \(W_1\cap W_1' \neq \{0\} \). Then \(W_1'\) is equivalent to \(W_1\).
\end{lemma}

\begin{proof}
The statement is obvious if \(W_1=W_1'\). Next suppose \(a_1\in V\) is a primitive element such that \( W_1\cap W_1' = \Z a_1\), \(b_1\in W_1\) such that \( W_1 = \Z a_1 \oplus  \Z b_1\). Let \(a_2\in W_1^\perp\) be a primitive element  such that \(W_1' = \Z a_1 \oplus  \Z (b_1 +\alpha _2 a_2)\) for some \(\alpha_2\in \Z\setminus 0\).

Complete \(a_,b_1,a_2\) into a symplectic basis \( a_1, b_1, \ldots, a_g, b_g\). We have \( (W_1\oplus W_1')^\perp = \Z a_2 \oplus   \sum _{i\geq 3} (\Z a_i \oplus  \Z b_i)\).  Notice that we have the following indeterminacy on the choice of basis: given elements  \(m_i, n_i \in \Z\) for \(i\geq 3\), define 
\[ a_1 '= a_1,\ b_1' = b_1,\ a_2' = a_2,\ b_2' = b_2 +\sum_{i\geq 3} n_i a_i -m_i b_i\]
and 
\[ a_i'= a_i + m_i a_2,\ b_i' = b_i +n_i a_2 \text{ for }i\geq 3. \]

We then have the symplectic decompositions \( V= W_1 \oplus  W_2 \oplus  W_3 \) and \( V= W_1'  \oplus  W'_2 \oplus  W_3\) with \( W_2=  \Z a_2  \oplus  \Z b_2 '\), \( W_2 ' = \Z a_2 \oplus \Z (\alpha_2 a_1 +b_2 ')\), and  \( W_ 3 = \sum _{i\geq 3} \Z a_i ' \oplus  \Z b_i' \).

\textit{Case 0:} If \(p( a_2 ) \neq 0\) it suffices to choose \( m_3\) such that \( p(a_3' ) \neq 0\). Then none of the restrictions of \( p\) to any of the modules \( W_1, W_2 , W_3 \) (resp. \( W_1 ', W_2 ', W_3\)) is trivial, hence \( W_1 \sim W_3 \sim W_1' \). \\

In the sequel we suppose $p(a_{2})=0$. For any choice of the \(m_{j}'s\) and \(n_{j}\)'s and any \(i\geq 3\) we have \(p(a_i')= p(a_i) \), \(p(b_i' )= p(b_i)\). \\

\textit{Case 1:} If \(p(W_{3})\neq 0\).\\
\textit{Sub-case 1.1:} If \(p(a_{1})=0\)  . Then, since \(W_{1}\) is \(p\)-admissible, \(p(b_{1})\neq 0\) and we can choose  \(m_i, n_i\in \Z\) appropriately to guarantee that \( p (b_2 ') \neq 0\). Since  \(p(\alpha_2 a_1 + b_2' )=p(b_{2}') \neq 0\) all submodules \(W_{1}, W'_{1}, W_{2}, W'_{2}, W_{3}\)  are \(p\)-admissible, so \(W_{1}\sim W_{3}\sim W_{1}'\).\\
\textit{Sub-case 1.2:} If \(p(a_{1})\neq 0\) .

\textit{Sub-subcase 1.2.1} If  \( p(W_{3}) \) contains at least three elements. We can still choose the \(m_i, n_i\in \Z\) appropriately to guarantee that  \( p(\alpha_2 a_1 + b_2' )\neq 0,\quad p(b_{2}') \neq 0\) to conclude as in the last case.\\

\textit{Sub-subcase 1.2.2:}  If \( p(W_{3}) \)  has at most two elements, i.e. it is either trivial or \(\mathbb{Z}/2\mathbb{Z}\). The intersection of \(\Z a_2 \oplus  W_3= (W_1+W_1')^\perp\) with the kernel of \(p\) contains a primitive rank two subgroup \(K\).  We then follow a route similar to that of Lemma \ref{l: creating intersection}. There exists an element in \(K\) which does not belong to \textit{ the union} of the spaces \( NA(p_{|(W_1)^\perp})\) and \(  NA(p_{|(W_1 ')^\perp})\) which have each rank bounded by one by Lemma \ref{l: admissible element}. Such an element belongs to a symplectic \(p_{|(W_1)^\perp}\)-admissible rank two submodule \( W_2 \subset (W_1)^\perp\), and to a symplectic  \(p_{|(W_1')^\perp}\)-admissible rank two submodule \( W_2 '\subset (W_1')^\perp\). We then have \(  W_2 \sim W_1\) and \( W_2' \sim W_1'\). By construction,  the intersection \(W_2\cap W_2'\) contains a non zero element in the kernel of \(p\) and \(v=a_1\) belongs to  \( (W_2 +W_2')^\perp\), and satisfies \(p(v)\neq 0\). By Case 0 or Sub-case 1.1 ,  we conclude that \( W_2 '\sim W_2\) and therefore \(W_1' \sim W_1\). 

\textit{Case 2:} If \(p(W_{3})=0\). We then have \(p(a_2)= p(a_i)= p(b_i) =0\) for \(i= 3,\ldots, g\)\\
\textit{Sub-case 2.1:} If \(p(a_{1})\neq 0\) we apply the same arguments of Sub-subcase 1.2.2\\
\textit{Sub-case 2.2:} If \(p(a_{1})=0\). Since \(W_{1}\) is \(p\)-admissible, \(p(b_{1}\neq 0\). In particular,  \( p(a_2)= p(a_i)=p(b_i)=0\) for \(i\geq 3\) and \(p(b_{1})\neq 0\).  Define \( W_1 '' = \Z a_1 \oplus  \Z (b_1 +  \alpha_2 a_2 + \sum _{i\geq 3} m_i a_i + n_i b_i)\), for some choices of \(m_i, n_i\) 's. The orthogonal of \( W_1'' \) contains the elements \[  a_2,  \alpha_2 a_1+ b_2 \text{ and for } i\geq 3 \text{ by } \alpha_2 a_i +n_i b_2 , \alpha_2 b_i -m_i b_2 \]
So since \(p(b_1+\alpha_2 a_2 + \sum _{i\geq 3} m_i a_i + n_i b_i) =p(b_1) \neq 0\) and \(p(\alpha_2 a_1+ b_2)=p(b_2)\neq 0\), \(W_1 ''\) is \(p\)-admissible. 

The module \( (W_1 + W_1'')^\perp \) contains  \(\alpha_2 a_3 + n_3 b_2\) for instance, so if \(n_3=1\), the restriction of \(p\) to \((W_1 + W_1'')^\perp \) does not vanish identically. In particular, by Cases 0 or 1 we conclude that \( W_1'' \sim W_1\). On the other hand, \((W_1')^\perp\) is the module generated by \( a_2, \alpha_2 a_1+b_2\) and the \(a_i, b_i\) 's for \(i\geq 3\), so \( (W_1'  + W_1'')^\perp \) contains \(\alpha_2 a_1 + b_2\) which does not belong to the kernel of \(p\). Again by  Cases 0 or 1 shows that \( W_1'' \sim W_1 '\). 

This finishes the proof.\end{proof}

\begin{proof}[Proof of Theorem \ref{t:graph of admissible}] By Lemma \ref{l: creating intersection}, any pair of \(p\)-admissible symplectic modules is equivalent to a pair of \(p\)-admissible modules that intersect on a non trivial subspace. We then deduce from Lemma \ref{l: proof of equivalence I}  that such a pair are equivalent, showing that the original two modules were equivalent as well. Hence the graph of \(p\)-admissible symplectic submodules of rank two is connected. The fact that the whole graph is connected comes from Corollary \ref{c: reduction to rank two}.
\end{proof}
\section{Proofs of main Theorems}



\begin{proof}[Proof of Theorem \ref{t:Cgn}: ]
By induction on the genus. The cases $g=1$ and $g=2$ are treated in Corollaries \ref{c:genus one case} and \ref{c: connectedness genus two}. Suppose $g\geq 3$ and Theorem \ref{t:Cgn} is true up to genus $g-1$. Let $p\in H^1(\Sigma_g,\mathbb{C}/\mathbb{Z})$ with $\deg(p)\geq 3$.

By Proposition \ref{p:degenerating to simple boundary point}, any point of $\per_g^{-1}(p)$ is equivalent to a simple boundary point with precisely one node. The homomorphism $[p]:H_1(\Sigma_g)\rightarrow\frac{\Z}{\frac{1}{2}\Z}$ is non-trivial by hypothesis. By Theorem \ref{t:graph of admissible} the graph of $[p]$-admissible decompositions is connected. Therefore, we can apply Proposition \ref{p:algebrization} to deduce that all simple boundary points are equivalent. In other words, the bordification of $\per^{-1}(p)$ is connected. By Corollary \ref{c:conn of closure equiv to conn} we deduce that $\per^{-1}(p)$ is connected. 
\end{proof}

\begin{proof}[Proof of Theorem \ref{t:leaf closures}:]
Apply the Transfer Principle (\cite{CDF2}) to transfer the properties of the action of the mapping class group on the image of $\per$ to properties of the isoperiodic foliation. The analysis of the action of the mapping class group of $\Sigma_g$ on $H^1(\Sigma_g,\mathbb{C}/\mathbb{Z})$ can be found in \cite[Proposition 4.2]{G} where Ratner's Theory is used to provide the described classification of the closures of the orbits. Moore's Theorem allows to deduce the ergodicity property.  
The closedness of the leaves with discrete $\Lambda$ is easy. If $\Lambda$ is infinite then the leaf cannot be algebraic, with respect to the Deligne-Mumford algebraic structure on moduli space. More precisely, let \(\Omega \overline{\mathcal M _{g,2}}  \) be the Hodge bundle over the Deligne-Mumford compactification.  In the boundary component \(\partial \subset \Omega \overline{\mathcal M _{g,2}}  \) consisting of forms having a node separating the curve in a curve of genus \(g\) and a rational curve containing the two marked points, let us consider the Zariski open subset \(\partial ^*\) defined by the curve of genus \(g\) being smooth. Given any lift \(P\) of \(p\) satisfying Proposition \ref{p: lifts}, let us introduce the subset \( \mathcal H_P\subset \partial^* \) consisting of  nodal forms constructed by attaching the form \( (\mathbb P^1, \frac{dz}{2i\pi z} ) \) with an abelian differential having a period equal to \(P\) after conveniently marking it.  Since the image of \(P\) is a lattice, the set \( \mathcal H_P\) is a copy of the a Hurwitz space of ramified coverings of degree \( \text{deg} (P):= \frac{\Re P \cdot \Im P }{\text{vol} (\C / \text{Im} (P)) } \) over the elliptic curve \( \C / \text{Im} (P)\). We claim that 
\[\overline{L}\cap \partial^* = \bigcup_{P \text{ lift of } p} \mathcal H _P , \]
which explains the reason why \(L\) is not algebraic since a union of a countable number of disjoint algebraic suvarieties is not algebraic. The inclusion \( \subset\) is merely the fact already mentioned that the limit of a sequence of elements of \(L\) is a nodal meromorphic differential whose period is equal to \(p\) in a certain marking. The inclusion \( \supset\) comes from the fact that one can deform isoperiodically a form in any element of \(\partial^*\) in the smooth part, together with the fact that, by Theorem \ref{t:Cgn}, the leaf \(L\) is exactly the set of element in \(\Omega \mathcal M_{g,2}\) having period \(p\) in a certain marking.
\end{proof}
\section{Appendix: The isoperiodic foliation on moduli spaces of stable meromorphic forms with simple poles}
In this appendix we review, adapt and extend Section 3 of \cite{CDF2} to the case of moduli spaces of stable meromorphic forms with simple poles.

\subsection{Augmented Teichm\" uller space and its stratification}
Recall the definition of the Augmented Teichm\" uller space from subsection \ref{ss:augmented teich definitions}
\begin{definition}
The augmented Teichm\" uller space $\overline{\mathcal{T}}_{g,n}$ is the set of all homotopically marked stable genus $g$ curves with $n$ marked points  up to equivalence. 
\end{definition}

The Teichm\" uller space $\mathcal{T}_{g,n}$ is the subset of  $\overline{\mathcal{T}}_{g,n}$ formed by curves without nodes. Its complement, denoted by $\partial\mathcal{T}_{g,n}=\overline{\mathcal{T}}_{g,n}\setminus \mathcal{T}_{g,n}$ is called the boundary. Given a subset $\mathcal{K}$ of $\overline{\mathcal{T}}_{g,n}$ its boundary is defined as $\partial \mathcal{K}:= \mathcal{C}\cap\partial\mathcal{T}_{g,n}$.

\begin{definition}
A curve system $c=\sqcup c_i$ in $\Sigma_{g,n}=(\Sigma_g,p_1,\ldots,p_n)$ is a disjoint collection of simple closed curves $c_i$ on $\Sigma_g\setminus \{q_1,\ldots,q_n\}$ none of which is isotopic to any other, to a point or to a cylinder in $\Sigma_g\setminus \{p_1,\ldots,p_n\}$. To a curve system $c$ we can associate the subset $B_{c}\subset \partial\mathcal{T}_g$ of the boundary consisting of homotopically marked stable curves topologically equivalent to a collapse $\Sigma_g\rightarrow \Sigma_g/c$ obtained by identifying each curve in $c$ to a distinct point (careful, we do not mean all curves in $c$ to a point) 
\end{definition}

Given a curve system $c$, for each component $\Sigma^{i}$ of $\Sigma_{g}\setminus c$ we define $(\Sigma_{g_i},P_i)$ to be the closed surface of genus $g_i$ with a set of marked points $P_i$ obtained by collapsing each boundary component of $\Sigma^i$ to a (marked) point and keeping the marked points of $\Sigma_{g}$ lying on $\Sigma^i$ in $Q_i$. By using attaching maps there is a natural identification  \begin{equation}B_{c}\cong\Pi_{i}\mathcal{T}_{g_i,n_i}.\label{eq:curvesystemdecomposition}\end{equation}
 
The boundary $\partial\mathcal{T}_{g,n}$ is the disjoint union of all boundary strata $\sqcup_{c}B_{c}$ where 
$c$ varies in the set of nonempty curve systems. $\mathcal{T}_{g,n}$ corresponds to the empty curve system. 
Each stratum $B_c$ has a topology and complex structure given by the bijection (\ref{eq:curvesystemdecomposition}). In particular it is connected.  All the curves that appear in a stratum have the same number of separating and non-separating nodes. When all the nodes are non-separating we say that the stratum is of compact type. 

Given a simple closed curve $c'\subset \Sigma_{g,n}$ we denote by $$D_{c'}=\bigsqcup_{c'\subset \curvesystem}B_{\curvesystem}$$ the union of all strata that collapse $c'$ to a node.

\subsection{Topology on $\overline{\mathcal{T}}_{g,n}$}
\label{ss:topology of Augmented teichm} 

For a detailed description of the topology we refer to \cite{ACG}[pp. 485-493] and references therein. 




The restriction of the given topology to $\mathcal{T}_{g,n}$ produces the so-called conformal topology. Abikoff showed (in \cite{Abikoff}[Theorem 1]) that this topology is equivalent to the Teichm\"uller topology. 

The restriction of the topology to the boundary set $B_{c}$ corresponding to a curve system $c$ is equivalent to the product topology obtained from (\ref{eq:curvesystemdecomposition}).



The topology is not locally compact around any boundary point. Indeed, if $U$ is a neighbourhood of a point in $B_c$, the action of the Dehn twist $\Delta_a:\Sigma_g\rightarrow \Sigma_g$ around a simple closed curve $a\in c$ fixes all the points in $U$ for which the marking collapses $a$ to a point, but has infinite orbits at any other point in $U$, see Remark \ref{rem:homotopically deform to dehn}. Therefore, there is no manifold structure in $\overline{\mathcal{T}}_{g,n}$ compatible with the given topology.
\begin{definition}
Given a curve system $c$ the distinguished neighbourhood of the stratum $B_{c}$ is the set   $$U_{c}=\bigsqcup_{c'\subset c}B_{c'}.$$
\end{definition}
We think of $U_c$ as the union of $\mathcal{T}_{g,n}$ with some of the boundary strata. These open sets will be useful to define and work with the complex structure on quotients of the Augmented Teichm\" uller space. 
\subsection{Complex structure and Deligne-Mumford-Knudsen compactification}
The mapping class group of $\Sigma_{g,n}$, i.e. the group $\text{Mod}(\Sigma_{g,n})$ of isotopy classes of orientation preserving diffeomorphisms that fix each marked point, 
acts on $\overline{\mathcal{T}}_{g,n}$ by homeomorphisms that preserve the stratification and are holomorphic in restriction to any stratum. The action is defined by pre-composition on the marking. The quotient $$\overline{\mathcal{M}}_{g,n}=\overline{\mathcal{T}}_{g,n}/\text{Mod}(\Sigma_{g,n})$$ is a compact topological space. 
It can be endowed with a complex orbifold structure. 

Consider a curve system $c$ and define  $\Gamma_c$ the abelian group generated by Dehn twists around the curves in $c$. Following \cite{Bers1}, the quotient $U_c/\Gamma_c$ is equivalent to a bounded domain in $\mathbb{C}^{3g-3}$. Under this equivalence, each stratum $B_{c'}$ associated to a simple closed curve $c'\subset c$ has image contained in a regular divisor $D_{c'}$. These divisors intersect normally and their intersections define the other different strata: the stratum associated to $c'\subset c$ is the intersection of all the divisors associated to the simple curves in $c'$. The complement of this divisor is the stratum $B_{\emptyset}=\mathcal{T}_{g,n}$ formed by smooth marked curves. 

The union of all natural maps $U_c/\Gamma_c\rightarrow \overline{\mathcal{M}}_{g,n}$ induce a system of (orbifold) charts with holomorphic transition maps on $\overline{\mathcal{M}}_{g,n}$. 

The boundary $\partial\mathcal{M}_{g,n}=\overline{\mathcal{M}}_{g,n}\setminus \mathcal{M}_{g,n}$ is a normal crossing divisor each of whose components correspond to the image of one of the  $D_c$'s of $\overline{\mathcal{T}}_{g,n}$ in the quotient $\overline{\mathcal{M}}_{g,n}$.

\subsection{The subgroup of \(\text{Mod}(\Sigma_{g,n})\) acting trivially on punctured homology}

Let $g,n\in \mathbb{N}$. Choose an oriented closed genus $g$ reference surface $\Sigma_g$ with a set of $n$ ordered distinct marked points $P=(p_1,\ldots,p_n)$. Denote by $\Sigma_{g,n^*}$ the (possibly nonclosed) surface  $\Sigma_{g}\setminus P$. We will use the following notation for certain relative homology groups $$H_1(\Sigma_{g,n^*})=H_1(\Sigma_g\setminus P,\mathbb{Z})$$
that we call punctured homology groups. The mapping class group $\text{Mod}(\Sigma_{g,n})$ induces an action on $H_1(\Sigma_{g,n^*})$ that fixes every element in the peripheral module, i.e, the module $\Pi_n$ generated by cycles turning positively once around each puncture $p_i$ for $i=1,\ldots, n$. 

The subgroup of $\text{Mod}(\Sigma_{g,n})$ that acts trivially on $H_1(\Sigma_{g,n^*})$ will be denoted by $\mathcal{I}(\Sigma_{g,n^*})$ and called the (punctured) Torelli group of $\Sigma_{g,n^*}$. For future reference we explicit the associated exact sequence of groups \begin{equation}
    \label{eq:exact seq torelli relative groups}0\rightarrow \mathcal{I}(\Sigma_{g,n^*})\rightarrow \text{Mod}(\Sigma_{g,n})\rightarrow \text{Aut}(H_1(\Sigma_{g,n^*}),\cdot)
\end{equation}
where the product preserved in homology is the intersection product. The action of the Dehn twist $\Delta_{\gamma}$ around a simple closed curve $\gamma$ in $\Sigma_{g,n^*}$ in $H_1(\Sigma_{g,n^*})$ is the module morphism defined by \begin{equation}\label{eq:dehn twist in homology}a\mapsto a+(a\cdot [\gamma])[\gamma]\end{equation}
The quotient will be denoted $$\overline{\mathcal{S}}_{g,n^*}:=\overline{\mathcal{T}}_{g,n}/\mathcal{I}(\Sigma_{g,n^*})$$ is what we call the augmented Torelli space of $\Sigma_{g,n^*}$. A point in this space will be denoted by a triple $(C,Q, [f])$ where $[f]$ denotes the equivalence class of the homotopical collapse map ${f}$ under the action of the corresponding Torelli group.  To each such point there corresponds an exact sequence  \begin{equation}0\rightarrow \ker [f]_*\rightarrow H_1(\Sigma_{g,n^*})\xrightarrow{[f]_*} H_1(C\setminus Q,\mathbb{Z})\rightarrow 0\label{eq:exseq homol marking}\end{equation} where $\ker [f]_*$ is the isotropic subgroup generated by the homology classes induced by the curves in the curve system $c$ collapsed by $f$ to the nodes of $C$. 

The action of the Torelli group on augmented Teichm\" uller space preserves the stratification. The class of each stratum $B_c$ in the quotient is characterized by the equivalence class $\curvesystem$ of curve system $c$ under the action of the Torelli group and denoted $B_{\curvesystem}$.

The open sets $U_c$ defined at the level of augmented Teichm\" uller space induce open sets that we denote $U_{\curvesystem}$ that are characterized by the Torelli class $\curvesystem$ of curve system $c$. The action in $U_{\curvesystem}$ of the group $\Gamma_{c}$ generated  by Dehn-twists around the curves in the curve system  coincides with that of the group generated only by non-separating curves of the curve system. Indeed, the Dehn twists around separating curves of $c$ define elements in $\mathcal{I}(\Sigma_{g,n^*})$ (even when the separating curve induces a non-trivial element in $H_1(\Sigma_{g,n^*})$).  The quotient $U_{\curvesystem}/\Gamma_{\curvesystem}$ is a complex manifold. In particular $U_{\curvesystem}$ is a complex manifold whenever all curves in $\curvesystem$ are separating. 

\subsection{The stratification of $\mnc_{g,n^{*}}$ and its dual boundary complex}

\begin{lemma}
The stratification induced on $\mnc_{g,n^{*}}$ by that of $\overline{\mathcal{T}}_{g,n}$ is a local abelian ramified covering of a normal crossing divisor and its dual boundary complex  $\mathcal{C}(\mnc_{g,n})$ is isomorphic to the quotient $\mathcal{C}_{g,n}/\mathcal{I}(\Sigma_{g,n^*})$ of the curve complex $\mathcal{C}_{g,n}$ of the surface $\Sigma_{g}$ deprived of the $n$ marked points.
\end{lemma}
\begin{proof}
 The stratification of the open set $U_{c}$ is invariant under the action of the subgroup $\Gamma_{c}\cap\mathcal{I}(\Sigma_{g,n^{*}})\subset \Gamma_{c}$ and a local abelian ramified cover of a normal crossing divisor. Therefore, the quotient  on $U_{c}/(\Gamma_{c}\cap\mathcal{I}(\Sigma_{g,n^{*}}))$ has the same local property. On the other hand, as a consequence of the fact that any automorphism of a stable curve that acts trivially on homology is equivalent to some Dehn twist around the pinched curves, $U_{c}/(\Gamma_{c}\cap\mathcal{I}(\Sigma_{g,n^{*}}))\rightarrow U_{c}/\mathcal{I}(\Sigma_{g,n^{*}})\subset \mnc_{g,n^{*}}$ is a local homeomorphism at every point of the stratum $B_{c}$. 
 A connected component of stratum of $\mnc_{g,n^{*}}$ corresponds to one orbit of connected components of a stratum of $\overline{\mathcal{T}}_{g,n}$ under the action of $\mathcal{I}(\Sigma_{g,n^{*}})$. The isomorphism of complexes then  follows from the isomorphism $\mathcal{C}(\overline{\mathcal{T}}_{g,n})\simeq \mathcal{C}_{g,n}$. 

\end{proof}

\subsection{Moduli of stable forms with simple poles and its substratification} 

Let us denote by $\Omega\overline{\mathcal{M}}_{g,n}\rightarrow\overline{\mathcal{M}}_{g,n}$ the bundle whose fiber over a point $(C,Q)$ is the vector space of meromorphic stable forms on the marked curve $(C,Q)$ having at worst simple poles on the $n$ distinct ordered points of $Q$. In other words the space of meromorphic sections of the twisted line bundle $K_C(-(q_1+\ldots+q_{n}))$. A point in $\Omega\overline{\mathcal{M}}_{g,n}$ will be denoted by $(C,Q,\omega)$. When $(\omega)_{\infty}=q_1+\ldots+q_{n_1}$ we will omit $Q$ and write  $(C,\omega)$.  
In the next Lemma we analyze the domain and analytic properties of functions defined by integrating stable forms along certain paths: 

\begin{lemma}\label{l:general holomorphic integral}
Let $c$ be a curve system in $\Sigma_{g,n^*}$ and $U_c/\Gamma_c\rightarrow U\subset\overline{\mathcal{M}}_{g,n}$ a distinguished (orbifold) chart. Given  $\gamma$ a path in $\Sigma_g\setminus \{c,P\}$ (avoinding possible nodes and poles), the map $\Omega U_c\rightarrow \mathbb{C}$ defined by $$(C,Q,\{f\},\omega)\mapsto \int_{f_*(\gamma)}\omega\in\mathbb{C}$$ where $f$ is a representative in its isotopy class that collapses some curves of $c$, is well defined and invariant by the action of $\Gamma_c$. It induces a \text{holomorphic} map $\Omega (U_c/\Gamma_c)\rightarrow \mathbb{C}$.
\end{lemma}

\begin{proof}
The map is well defined on $U_c/\Gamma_c$ because $\gamma$ does not intersect any of the simple closed curves in the curve system $c$ that is collapsed to the nodes of the curve. It is holomorphic around any point with finite value outside the boundary and bounded in the neighbourhood of every boundary point with finite value. On the other hand, we know that the boundary forms a divisor. Hence  by Riemann's extension Theorem we deduce that there is a unique holomorphic extension to the boundary. 
\end{proof}


The boundary stratification of the bundle $\Omega\overline{\mathcal{M}}_{g,n}$ by the number of nodes is substratified by the orders of the form at zeros and nodes. The order of $\omega$ at a regular point $q\in C$ is defined to be the $\text{ord}_q(f)$ where $\omega(z)=f(z)dz$ in a holomorphic coordinate $z:(C_1,q)\rightarrow \mathbb{C}$ around $q$. The order of $\omega$ at a node $q\in C$ is $$\text{ord}_q(\omega)=2+\text{ord}_q(\omega_{|C_1})+\text{ord}_q(\omega_{|C_2})$$ where $C_1$ and $C_2$ are the branches of $C$ at $q$. The order of the form at any point is clearly invariant by biholomorphism. 

\begin{definition}Two stable forms belong to the same substratum if there is a homeomorphism between the underlying curves with marked points preserving the order of the forms at each point (nodes to nodes, zero components to zero components,  isolated zeros (resp. poles) are preserved with the same order.
\end{definition}
Some strata admit a distinguished atlas produced by integration and a use of the Gauss-Manin connection to identify homology groups: 
\begin{lemma}\label{l:period coordinates}
Let  $(C,Q,\omega)\in \Omega\overline{\mathcal{M}}_{g,n}$ be a stable meromorphic form $\omega$ on a marked stable curve $(C,Q)$ having simple poles on $P(\omega)=Q$, discrete zero set $Z(\omega)$ and zero residues at the nodes $N(C)$. Then, integration on cycles produces a well defined map from a neighbourhood $R$ in its  stratum  of $1$-forms  
\begin{equation}\label{eq:restriction to stratum} R\rightarrow \operatorname{Hom}(H_1(C\setminus P(\omega),N(C)\cup Z(\omega);\mathbb{Z});\mathbb{C})\end{equation}
that is a homeomorphism onto its image.

\end{lemma}
\begin{proof}
The result is well known if $C$ is smooth. (\cite{GK1,GK2}) There is a slight generalization in this case that will be useful for our purposes: the case where $C\setminus Q$ has also a set $A$ of $m$ marked points  (that can coincide or not with the zeros of $\omega$). Integration allows to define a map on the neighbourhood $R$ of $(C,Q\cup A, \omega)$ in its stratum of  $\Omega\mathcal{M}_{n+m}$ namely 
\begin{equation}
    R\rightarrow\text{Hom}( H_1(C\setminus Q, Z(\omega)\cup A, \mathbb{Z}); \mathbb{C}).
\end{equation}
It is also a homeomorphism onto its image.

Let us assume $C$ has some node.  Consider the normalization $\hat{C}$ marked by the pairs $N_j=\{n_j^+,n_j^-\}$ of points that produce the $j$'th node of $C$. Denote $N=\cup N_j $ and  $C_1,\ldots,C_k$ the (smooth) components of the normalization $\hat{C}$, $\omega_1,\ldots,\omega_k$ the restriction of $\omega$ to $C_j$, $Q_j=C_j\cap Q$ and $A_j=C_j\cap N$. Let $R$ denote the neighbourhood of $(C,Q,\omega)$ in its stratum and $R_j$ denote the neighbourhood of the point $(C_j,Q_j\cup A_j,\omega_j)$ in its stratum.  Up to reducing $R$ we can construct a homeomorphism  $$R_1\times\cdots\times R_k\rightarrow R$$
by applying the rule of attaching map given by that of $C$. Since each $(C_j,Q_j\cup A_j,\omega_j)$ is a non-zero form on a smooth compact curve with marked points we can apply the Lemma in the case of a smooth curve with marked points to deduce that, up to reducing the size of the $R_j$'s, we have a homeomorphism $$ R_1\times\ldots,\times R_k\simeq\bigoplus_{j=1}^k \text{Hom}(H_1(C_j\setminus P(\omega_j), Z(\omega_j)\cup A_j,\mathbb{Z});\mathbb{C}).$$
The latter is isomorphic to $$\text{Hom}\Big(\big( \bigoplus_{j=1}^k H_1(C_j\setminus P(\omega_j), Z(\omega_j)\cup A_j,\mathbb{Z})\big);\mathbb{C}\Big).$$ To finish the proof, we claim that there exists a natural isomorphism \begin{equation}\label{eq: relative homology pairs}\bigoplus_{j=1}^k H_1(C_j\setminus P(\omega_j), Z(\omega_j)\cup A_j,\mathbb{Z})\simeq H_1(C\setminus P(\omega),N(C)\cup Z(\omega),\mathbb{Z}).\end{equation} 
We sketch the last equivalence in the case of one node, and leave the recurrence argument for the reader. Suppose that $C$ has one node $N_1$ and its normalization one or two components $C_j$'s. Let $A_1$ the marked points in the normalization corresponding to $N_1$. The relative homology of the pair $\big((\sqcup C_j\setminus P(\omega_j))/A_1, \{A_1\}\cup Z(\omega) \big)$ is isomorphic to  the sum of relative homologies of the left hand side of \eqref{eq: relative homology pairs}. On the other hand the pair is homeomorphic to the  $(C\setminus P(\omega),N(C)\cup Z(\omega))$.
\end{proof}
\subsection{Homology coverings of bundles of stable forms and period map}
The pull back of the bundle $\Omega\overline{\mathcal{M}}_{g,n}$  by the (branched) cover $\overline{\mathcal{S}}_{g,n^*}\rightarrow \overline{\mathcal{M}}_{g,n}$ will be denoted  $$\Omega\overline{\mathcal{S}}_{g,n^*}\rightarrow\overline{\mathcal{S}}_{g,n^*}$$
A point in $\Omega\overline{\mathcal{S}}_{g,n^*}$ will be denoted by $(C,Q,[f],\omega)$. We want to describe the set of forms in this space for which all integrals on classes of $H_1(\Sigma_{g,n^*})$ are well defined complex numbers that coincide with those of some form on a \textit{smooth} curve. A problem that arises is that stable forms can have non-zero residues at nodes, and this does not allow to integrate any path passing through the node.
The Mayer-Vietoris exact sequence can be used to overcome this difficulty in the case of separating nodes.  Indeed, consider a family of separating curves $c_1,\ldots,c_{k}$ of some curve system $c$ and write $\Sigma_g\setminus \{c_1\cup\cdots\cup c_{k}\}=\Sigma^{(1)}\sqcup\ldots\sqcup\Sigma^{(\nu)}$, where each $\Sigma^{(i)}$ is connected and non-empty.  Each $\Sigma^{(i)}$ can be identified with some  $\Sigma_{g_i,n^*_i}$, a genus $g_i$ compact surface with $n_i$ punctures (we think of a boundary component as a point). For each $c_i$ we consider two peripheral curves \(\gamma_i^-,\gamma_i^+\) on each of its sides. Then \begin{equation}\label{eq:mayervietoris} H_1(\Sigma_{g,n^*})\cong\frac{H_1(\Sigma_{g_1,n^*_{1}})\oplus\cdots\oplus H_1(\Sigma_{g_{\nu},n^*_{\nu}})}{<[\gamma_i^+]+[\gamma_{i+1}^-]:i=1,\ldots,k-1>}\end{equation} where the class $[\gamma_i^-]$ (resp. $[\gamma_i^+]$) is in the homology group of the (punctured) component to which it belongs.

We deduce that every homology class in $H_1(\Sigma_{g,n^*})$ can be written as a sum of classes disjoint from the separating nodes. However this is not true for non-separating nodes. To be able to integrate a stable form on any class in $H_1(\Sigma_{g,n^*})$ we need to suppose that the residues are zero at non-separating nodes. The union of strata with this property will be useful for our purposes: $$\Omega_0\overline{\mathcal{S}}_{g,n^*}=\{(C,Q,[f],\omega)\in\Omega\overline{\mathcal{S}}_{g,n^*}: \text{Res}_q(\omega)=0\quad\forall \text{ non-separating node } q \text{ of } C\}.$$
Unfortunately this set is neither open (its interior is formed by stable forms on curves of compact type) nor closed (it is dense) in $\Omega\overline{\mathcal{S}}_{g,n^*}$. However it is the set where forms can be integrated in \textit{all} classes of $H_1(\Sigma_{g,n^*})$
\begin{definition}
Given $(C,Q,[f],\omega)\in\Omega_0\overline{\mathcal{S}}_{g,n^*}$, we define an element $p=\text{Per}(C,Q,[f],\omega)\in\text{Hom}(H_1(\Sigma_{g,n^*});\mathbb{C})$ called the period homomorphism of $(C,Q,[f],\omega)$ as follows: 
\begin{itemize}
    \item Consider the decomposition as in \eqref{eq:mayervietoris} using all separating nodes of $C$. 
    \item  Define  the homomorphism $$\bigoplus_{i=1}^{\nu}H_1(\Sigma_{g_i,n_{i}^*})\ni a\mapsto \int_{[f]_*(a)}\omega\in\mathbb{C}.$$ 
    \end{itemize}
    The opposite residue condition of stable forms at nodes implies that the homomorphism defines a homomorphism $p$ on the quotient on the right hand side of \eqref{eq:mayervietoris} that can be thought as an element in $\text{Hom}(H_1(\Sigma_{g,n^*});\mathbb{C})$.
 
\end{definition}

\begin{definition}
The period map on $\Omega_0\overline{\mathcal{S}}_{g,n^*}$ is the map $$\percompact_{g,n^*}:\Omega_0\overline{\mathcal{S}}_{g,n^*}\rightarrow \text{Hom}(H_1(\Sigma_{g,n^*});\mathbb{C})$$
sending each point to its period homomorphism. 
\end{definition}

By the decomposition in \eqref{eq:mayervietoris}, and  Lemma \ref{l:general holomorphic integral} the period map is holomorphic in the neighbourhood of a marked stable form on a curve of \textit{compact type}. Moreover if the residues at the nodes are all zero and the zeros of the form are isolated, the restriction of the period map to the stratum, written in the coordinates of Lemma \ref{l:period coordinates}, is just a linear projection. This implies that the period map is submersive in the neighbourhood of such a point (even in restriction to the stratum or to any smooth manifold containing it). In particular the local fiber at such a point is a regular holomorphic manifold transverse to any boundary component passing thorough it. 


This description of the local fibers of the period map is also true in more generality (possibly up to a branched cover) for forms with isolated zeros.
\begin{theorem}\label{t:local structure}

The local fiber of the period map in $\Omega_{0}\overline{\mathcal{S}}_{g,n^*}$  at a point $(C,Q,[f],\omega)\in \Omega_0\overline{\mathcal{S}}_{g,n^*}$ with isolated zeros, projects to the orbifold chart of $\Omega\overline{\mathcal{M}}_{g,n}$ as a complex manifold transverse to all boundary divisors through the point. Therefore, it is  an abelian ramified cover of a normal crossing divisor in $(\mathbb{C}^{2g+n-3},0)$ having precisely one component of codimension one in each component of codimension one of the ambient space through the point. \end{theorem}

\begin{proof}
The local period fiber on $\Omega_{0}\mnc_{g,n^{*}}$ at  $(C,Q,[f],\omega)\in \Omega_0\overline{\mathcal{S}}_{g,n^*}$ can be lifted to a local period fiber $L\subset \Omega_{0}\overline{\mathcal{T}}_{g,n}$ of the map \begin{equation}\Omega_0\overline{\mathcal{T}}_{g,n}\rightarrow H^1(\Sigma_g,q_1,\ldots,q_n,\mathbb{C})\end{equation} that associates to any  stable form with zero residues at non-separating nodes its period homomorphism. If the lift of $(C,Q,[f],\omega)$  belongs to the stratum $\Omega B_{c}$, the covering group is the subgroup $\Gamma_{c}\cap\mathcal{I}_{g}\subset \Gamma_{c}$. Hence, to prove the result we just need to prove that $L$  is locally an abelian ramified cover over a normal crossing divisor. 

Let $U_{c}$ denote the distinguished neighbourhood of $B_{c}$.  It suffices in fact to show that at the level of the quotient $\Omega_{0}(U_{c}/\Gamma_{c})$ the period fiber is a holomorphic manifold transverse to every boundary component passing through the point. Thanks to the normal crossing condition of the ambient space, this is guaranteed if the period map is submersive in restriction to the (regular) stratum of the normal crossing divisor where the point belongs to. In fact we will not use the boundary stratum but a smaller regular submanifold, consisting of the substratum through the point $(C,Q,\{f\},\omega)\in  L\subset \Omega_0\overline{\mathcal{T}}_{g,n}$.


Next we prove the analyticity of $L/\Gamma_{\curvesystem}$ in two steps corresponding to the dual exact sequence of \eqref{eq:exseq homol marking}, namely $$0\rightarrow\text{Hom}(H_1(C\setminus Q),\mathbb{C})\rightarrow \text{Hom}(H_1(\Sigma_{g,n^*}),\mathbb{C})\rightarrow \text{Hom}(\ker [f]_*,\mathbb{C})\rightarrow 0. $$ Remark that by definition the intersection of any pair of elements of $\ker[f]_*$ is zero. 

\begin{enumerate}

    \item[1)] Integration on the(peripheral) curves of $\curvesystem$ and a use of Riemann-Roch's Theorem as in subsection 3.16 in \cite{CDF2} we show that there is a well defined map $$\text{NRes}:\Omega(U_{\curvesystem}/\Gamma_{\curvesystem})\rightarrow \text{Hom}(\ker[f]_*;\mathbb{C})$$ that is holomorphic  thanks to Lemma \ref{l:general holomorphic integral}. It extends the residue map at the  nodes of $C$ for forms in $\Omega(C)$. The same application of Riemann-Roch's Theorem allows to show that the (linear) restriction of $\text{NRes}$ to $\Omega(C)$ is surjective. 
   Denote by $p=\percompact(C,Q,[f],\omega)$ the period homomorphism. By construction $\text{NRes}(C,Q,[f],\omega)=p_{|\ker[f]_*}$.  The base point belongs to  the smooth complex manifold $\text{NRes}^{-1}(p_{|\ker[f]_*} )\subset \Omega_0(U_{\curvesystem}/\Gamma_{\curvesystem})$. Denote $\widetilde{\text{NRes}^{-1}}(p_{|\ker[f]_*})$ the lift of  $\text{NRes}^{-1}(p_{|\ker[f]_*})$ to $\Omega U_{c}$. It  is contained in $\Omega_0 U_{c}$, and by construction it is a union of fibers of the period map.  
    \item[2)]
      The restriction of the period map to $\widetilde{\text{NRes}^{-1}}(p_{|\ker[f]_*})$ is well defined. It is $\Gamma_{\curvesystem}$-invariant, because for every curve $c_i$ of the curve system $c$ we have that either its intersection with any other class is zero (if $c_i$ is separating) or the period along $c_i$ of a form in $\widetilde{\text{NRes}}^{-1}(p_{|\ker[f]_*})$ is zero. Therefore it induces a holomorphic map \begin{equation}\label{eq:holo per}h:\text{NRes}^{-1}(p_{|\ker[f]_*})\rightarrow \text{Hom}_{p_{|\ker [f]_*}}(H_1(\Sigma_{g,n^*}),\mathbb{C})\end{equation}  with image in the set $\text{Hom}_{p_{|\ker [f]_*}}(H_1(\Sigma_{g,n^*}),\mathbb{C})$ of homomorphisms that extend $p_{|\ker [f]_*}$ to a homomorphism on $H_1(\Sigma_{g,n^*})$. The fibers of $h$ are analytic sets. Locally at the base point the fiber is $L/\Gamma_{\curvesystem}$. 
     \end{enumerate}
     Next we prove that under the hypothesis of isolated zeros of $\omega$ the set $L/\Gamma_{\curvesystem}$ is smooth and transverse to each boundary component. This is a consequence of the stronger
     
     \begin{lemma}\label{l:submersion on stratum}
     The map $h$ defined by \eqref{eq:holo per} restricted to the intersection with the substratum at a point with isolated zeros is submersive .
     \end{lemma}
      
     Hence the local fiber $H$ of the map \eqref{eq:holo per} at the base point is a smooth complex manifold transverse to every smooth manifold containing the substratum of the base point. In particular $H$ is transverse to each irreducible boundary component passing through the base point. We deduce that  $\partial H$ is a normal crossing divisor in $H$, and moreover $\partial H$, $H$ and $H\setminus\partial H$ are connected.


\begin{proof}[Proof of Lemma \ref{l:submersion on stratum}:] Let $R$ denote a neighbourhood in its stratum of the form $(C,Q,\omega)\in\Omega\overline{\mathcal{M}}_{g,n}$ that has zero residues at the non-separating nodes. Recall that, if a composition of maps is submersive then the last map of the composition is submersive also. We want to show that  $$h_{|R}:R\cap \text{NRes}^{-1}(p_{|\ker[f]_*})\rightarrow \text{Hom}_{p_{|\ker[f]_*}}(H_1(\Sigma_{g,n^*});\mathbb{C})$$ 
is submersive whenever $\omega$ has isolated zeros.

If all nodes of $C$ have zero residue for $\omega$ we can apply Lemma \ref{l:period coordinates} and find coordinates in $R$ where the map $h_{|R}$ is written as the linear projection $$\text{Hom}(H_1(C\setminus Q,N(C)\cup Z(\omega);\mathbb{Z});\mathbb{C})\rightarrow \text{Hom}(H_1(C\setminus Q,;\mathbb{Z});\mathbb{C})$$ induced by inclusion $H_1(C\setminus Q,\mathbb{Z})\rightarrow H_1(C\setminus Q,  N(C)\cup Z(\omega),\mathbb{Z})$.  Therefore it is submersive. 

Suppose there is some node with non-zero residue. 
Let $PN(C)=\{N_1,\ldots, N_{k-1}\}$ denote the nodes of $C$ with non-zero residue for $\omega$ (i.e. polar nodes). All of them are separating by hypothesis. Write $C\setminus PN(C)$ as a disjoint union $\sqcup_{j=1}^\nu (C_j\setminus PN(C_j))$ where $C_j$ is a stable curve  where the restricted form $\omega_j=\omega_{|C_j}$ has only zero residues at the nodes and simple poles on the set $Q_j=(Q\cup PN(C))\cap C_j$. Let $R_j$ be the stratum of the form $(C_j,Q_j,\omega_j)$ in $\Omega\overline{\mathcal{M}}_{g_j,n_j}$, where $g_j$ is the genus of $C_j$ and $n_j$ the cardinality of $Q_j$. Consider the subset  $R'_j\subset R_j$ of forms having the same residues at the poles $PN(C_j)$ as $\omega_j$ has.

The attaching rule to obtain $C$ from the $C_j$'s can be used to define a holomorphic map $$\varphi: R'_1\times\cdots\times R'_k\rightarrow R\cap\text{NRes}^{-1}(p_{|\ker[f]_*}).$$ In its image we find only forms in the stratum having the same residues at the nodes corresponding to $PN(C)$ as $\omega$ has. We claim that $$h\circ \varphi:R'_1\times\cdots\times R'_k\rightarrow \text{Hom}_{p_{|\ker [f]_*}}(H_1(\Sigma_{g,n^*});\mathbb{C})$$ is submersive. 
Recall from Lemma \ref{l:period coordinates} that in each $R_j$ we have holomorphic coordinates with values in $\text{Hom}(H_1(C_j\setminus P(\omega_j), N(C_j)\cup Z(\omega_j),\mathbb{Z});\mathbb{C}).$ In the given coordinates, the set $R_j'$ corresponds to a codimension one or two linar subspace, corresponding to fixing the residues at the poles $PN(C_j)$. As before, the projection  \begin{equation}\label{eq:projection}\pi_j: \text{Hom}(H_1(C_j\setminus P(\omega_j), N(C_j)\cup Z(\omega_j),\mathbb{Z});\mathbb{C})\rightarrow \text{Hom}(H_1(C_j\setminus P(\omega_j),\mathbb{Z});\mathbb{C})\end{equation} is linear, surjective.   The restriction of $\pi_j$ to $R_j'$ is submersive onto the subspace $$\text{Hom}_{PN(\omega_j)}(H_1(C_j\setminus P(\omega_j),\mathbb{Z});\mathbb{C})$$ of homomorphisms whose value around the peripherals of $PN(C_j)$ is the same as that of $\omega_j$. 
The  projections allow to define a map on the cartesian product \begin{equation}\label{eq:cartesian R's} R_1'\times\cdots\times R'_k\rightarrow \bigoplus_{j=1}^k\text{Hom}_{PN(\omega_j)}(H_1(C_j\setminus P(\omega_j),\mathbb{Z});\mathbb{C})\end{equation} that is submersive. 
By using a similar argument as in equation \eqref{eq:mayervietoris}, but only with the separating nodes of non-zero residue, we deduce that the target space of the map \eqref{eq:cartesian R's} is homomorphic to $\text{Hom}_{p_{|\ker[f]_*}}(H_1(\Sigma_{g,n^*});\mathbb{C})$. 
In this expression, the map $h\circ\varphi$ is the map \eqref{eq:cartesian R's}, so we have proven that
 $h\circ \varphi$ is submersive. Therefore, so are $h_{|R}$ and $h$. 

\end{proof}
\end{proof}

\begin{corollary}\label{c:connected closure}
Let  $F\subset\Omega\mathcal{S}_{g,n^*}$ be a fiber of $\per_{g,n^*}$ and $\overline{F}$ its closure in $\Omega^*_0\overline{\mathcal{S}}_{g,n^*}$. Then $F$ is connected if and only if $\overline{F}$ is connected. \end{corollary}

\begin{proof}

All the points in $\overline{F}$ have the same value for $\percompact_{g,n^*}$. Suppose $F$ is connected. Thanks to Theorem \ref{t:local structure}, the connected component of $\overline{F}$ at any boundary point contains points of $F$. Hence $\overline{F}$ is connected. For the converse, suppose $\overline{F}$ is connected. It is also path connected. Let $t\mapsto\gamma(t)$ be a path in $\overline{F}$ joining a pair of points of $F$. By  Theorem\ref{t:local structure} there are points of $F$ in any neighbourhood of any of the points $\gamma(t)$. All of them lie in the same component of $F$. Hence the starting and endpoint lie in the same connected component of $F$. This shows that $F$ is path connected.
\end{proof}

\end{document}